\documentclass[11pt, final]{amsart}
\usepackage{amsmath, amsxtra, amssymb, mathdots}
\usepackage{mathpazo}
\usepackage{verbatim}
\usepackage{multirow}

\usepackage[svgnames]{xcolor}
\usepackage{tikz}
\usepgflibrary{decorations.text}
\usetikzlibrary{decorations.text}
\pgfdeclarelayer{edgelayer}
\pgfdeclarelayer{nodelayer}
\pgfsetlayers{edgelayer,nodelayer,main}

\tikzstyle{none}=[inner sep=0pt]
\definecolor{hexcolor0xfefdfd}{rgb}{0.996,0.992,0.992}

\usepackage[margin=1.21in]{geometry}
\usepackage{graphicx}
\usepackage[cmtip,all]{xy}
\usepackage{xypic}

\usepackage[notcite, notref]{showkeys}
\usepackage[colorlinks, linkcolor=blue, citecolor=red, urlcolor=black]{hyperref}

\theoremstyle{plain}
\newtheorem*{main-theorem}{Theorem}
\newtheorem{theorem}[equation]{Theorem}
\newtheorem{prop}[equation]{Proposition}
\newtheorem{corollary}[equation]{Corollary}

\newtheorem*{claim*}{Claim}

\newtheorem{lemma}[equation]{Lemma}
\newtheorem*{condition}{Condition}

\theoremstyle{definition}
\newtheorem{definition}[equation]{Definition}
\newtheorem{example}[equation]{Example}
\newtheorem{remark}[equation]{Remark}
\newtheorem{conjecture}[equation]{Conjecture}
\newtheorem{question}[equation]{Question}

\numberwithin{equation}{subsection}

\DeclareMathOperator\spec{Spec}
\DeclareMathOperator{\Pic}{Pic}

\DeclareMathOperator{\Nef}{Nef}
\DeclareMathOperator{\NS}{NS}
\DeclareMathOperator{\Semiample}{Semiample}

\DeclareMathOperator{\card}{card}
\DeclareMathOperator{\Easy}{Easy}
\newcommand{\Supp}{\operatorname{Supp}}
\def\DD{\mathfrak{D}}
\def\ra{\rightarrow}
\def\fnef{\mathrm{F-nef}}
\def\dem{\mathrm{dem}}
\def\PS{\mathrm{PS}}
\def\PSFN{\mathrm{PSFN}}

\def\Sm{\mathfrak{S}}

\newcommand\M{\overline{M}}
\newcommand{\nhalf}{\lfloor n/2\rfloor}
\newcommand\modp[2]{\langle #1\rangle_{#2}}
\renewcommand{\sl}{\mathfrak{sl}}

\def\A{\mathcal{A}}
\def\B{\mathcal{B}}

\def\co{\colon\thinspace} 
\def\QQ{\mathbb{Q}}

\def\PP{\mathbb{P}}
\def\ZZ{\mathbb{Z}}
\def\cO{\mathcal{O}}

\begin{document}

\title{Semiampleness criteria for divisors on $\M_{0,n}$}

\begin{abstract}
We develop new characteristic-independent 
combinatorial criteria for semiampleness of divisors on $\M_{0,n}$.
As an application, we associate to a cyclic rational quadratic form satisfying a certain balancedness condition 
an infinite sequence of semiample line bundles.
We also give several sufficient and effective conditions for an $\Sm_n$-symmetric divisor on $\M_{0,n}$ 
to be semiample or nef. 

\end{abstract}
\author{Maksym Fedorchuk}
\maketitle

\setcounter{tocdepth}{1}
\tableofcontents

\section{Introduction}
Let $\M_{0,n}$ be the Deligne-Mumford moduli space. 
It is a smooth projective scheme
over $\spec \ZZ$, which in addition to representing the functor of stable $n$-pointed rational curves,
has misleadingly simple (but highly non-obvious)
blow-up constructions out of $\PP^{n-3}$ and $(\PP^1)^{n-3}$;
see \cite[p.546, and Theorems 1, 2]{keel} and \cite[Chapter 4]{kap1} for a precise functorial
definition and two different blow-up constructions of $\M_{0,n}$.

The Picard group of $\M_{0,n}$ is well-understood (see \S\ref{S:effective-criterion}).
However, many finer questions about the birational geometry of $\M_{0,n}$
remain highly elusive. 
In particular, the following two open questions about divisor classes 
on $\M_{0,n}$ form the context for the present work:
\begin{question}\label{Q1}
What are the nef divisors on $\M_{0,n}$?
\end{question}
\begin{question}\label{Q2}
Is every nef divisor on $\M_{0,n}$ semiample?
\end{question}
The famous \emph{F-conjecture} purports to give an answer to Question \ref{Q1} 
in terms of the F-curves, which are
the $1$-dimensional boundary strata in $\M_{0,n}$ (see \S \ref{S:boundary}).
To state it, recall that a line bundle on $\M_{0,n}$ is \emph{F-nef} if 
it intersects
every F-curve non-negatively.
\begin{conjecture}[
{\cite[Question 1.1]{KMc}, \cite[Conjecture 0.2]{GKM}}]
\label{F-conjecture}
A line bundle on $\M_{0,n}$ is nef if and only if it is F-nef. 
\end{conjecture}
In what follows, the \emph{F-cone} will be the set of F-nef divisor classes in $\NS(\M_{0,n})$.
Thus, the F-conjecture postulates that $\Nef(\M_{0,n})$
equals to the finite polyhedral F-cone. An analogous conjecture has also been formulated for $\M_{g,n}$
(see \cite[Conjecture 0.2]{GKM}, \cite[p.1228]{Gib}), but
the positive genus case follows from the genus $0$
case by a powerful result of Gibney, Keel, and Morrison 
\cite[Theorem 0.7]{GKM} (which holds in every characteristic besides $2$). 

The F-conjecture for $\M_{0,n}$  
has been verified only for $n\leq 7$ \cite{larsen}, and not much is known for higher $n$.
The situation for symmetric line bundles is slightly better. 
The \emph{symmetric F-conjecture}, which states that an $\Sm_n$-symmetric 
line bundle on $\M_{0,n}$ is nef if and only if it is F-nef, 
was proven by Gibney for all $n\leq 24$ \cite{Gib},
greatly extending an earlier work of Keel and McKernan, who used Mori-theoretic
techniques to prove the symmetric F-conjecture for $n\leq 11$ \cite{KMc}.
(Because of their reliance on the contraction theorem, the results of \cite{KMc} and \cite{Gib}
are restricted to characteristic $0$.) As a corollary, Gibney establishes the F-conjecture
for $\M_{22}$ and $\M_{24}$, the first two moduli spaces of curves known 
to be of general type.

Since every nef line bundle on a Mori Dream Space is semiample by definition \cite{hu-keel}, 
Question \ref{Q2} would have the answer ``All'' if the following question had
an affirmative answer:
\begin{question}[{\cite[Question 3.2]{hu-keel}}]
\label{Q:hu-keel}
Is $\M_{0,n}$ a Mori Dream Space?
\end{question}
The above question attracted a lot of attention and was sometimes formulated
as a conjecture until Castravet and Tevelev
proved that $\M_{0,n}$ is not a Mori Dream Space for $n\geq 134$ \cite{castravet-tevelev-MDS},
a result later improved to $n\geq 13$ by Gonz\'{a}lez and Karu \cite{gonzalez-karu}.
We note however that there is no known example of a nef, but not semiample divisor on $\M_{0,n}$;
and that it is still an open problem whether $\M_{0,n}/\Sm_n$ is a Mori Dream Space.

\subsection{Approaches to semiampleness}
\label{S:4-approaches}
Before we describe our results, we recall
that there are roughly four methods to proving semiampleness of a divisor on $\M_{0,n}$
(the first three of which are applicable to almost any projective variety). The first 
method is by proving ampleness using intersection theory. The second method is by 
proving nefness and then applying Kawamata's base point freeness
theorem. 
Both of these methods are illustrated in the proof of ampleness of $\psi-\Delta$ in \cite[Lemma 3.6]{KMc}.
The third method is by constructing regular morphisms from $\M_{0,n}$ to other varieties, 
which often have a modular interpretation (e.g., as GIT quotients), 
and by pulling back ample divisors from the target (see \cite{alexeev-swinarski,gian,GJMS}). Finally, Fakhruddin showed that in characteristic $0$ there is an enormous collection of 
semiample divisors on $\M_{0,n}$ coming from the theory of conformal blocks \cite{fakh}. 
In recent years, properties of conformal blocks divisors on $\M_{0,n}$
have been extensively studied in \cite{fakh, agss, ags, gian-gib, conf-block}. 
It is an open question to what extent the conformal blocks divisors cover $\Nef(\M_{0,n})$.

Even if we restrict ourselves to $\Sm_n$-symmetric divisors on $\M_{0,n}$, 
of the four methods described above, 
only the second can be considered effective in the sense that given a symmetric divisor $D$, 
there is an algorithm 
for determining whether the method applies to prove that $D$ is semiample. 
However, Kawamata's base point freeness theorem is restricted to characteristic $0$,
and still requires one to first prove that $D$ is nef. Even then, the cone of log canonical F-nef divisors,
that is, divisors that are semiample (assuming the F-conjecture) by Kawamata's base point freeness theorem, 
is well-known to be strictly contained inside the F-cone. 

\subsection{Our results}
\label{S:our-results}
In this work, we do not take any position on the open questions discussed above.
Instead, with the view toward these problems, 
we introduce a new combinatorial framework for proving semiampleness
(respectively, nefness) for a certain family of divisors on $\M_{0,n}$
over $\ZZ$ (respectively, over fields
of arbitrary characteristic). 

Our method is elementary in that it relies only on Keel's relations in $\Pic(\M_{0,n})$, 
but is also characteristic-independent, with most of the results holding over $\spec(\ZZ)$.
For $\Sm_n$-symmetric divisors we give several effective numerical semiampleness criteria. 
Namely, given a divisor $D=\sum_{r=2}^{\lfloor n/2\rfloor} a_r\Delta_r$ on $\M_{0,n}$,
our criteria verify semiampleness of $D$ by checking whether its coefficients 
satisfy finitely many linear inequalities  (see Theorem \ref{T:semiample}) or finitely many quadratic 
inequalities (see Theorem \ref{T:2nd-semiample}). The  
running times of the algorithms implementing these criteria 
vary between $o(2^n)$ and $o(n^4)$.

Because of the effectivity of our criteria, and because semiample divisors that we produce often 
lie on the boundary of $\Nef(\M_{0,n})$ and are generally not log canonical,
our method is complementary to the four methods described in \S\ref{S:4-approaches}. 
There is also a non-trivial overlap. 
For example, the fact that the quadratic form  
\[
x_0^2+x_1^2+\cdots+x_{m-1}^2
\]
is balanced in the sense of Definition \ref{D:balanced} gives a new, 
characteristic-independent 
proof of semiampleness of all $\sl_m$ level $1$ conformal blocks divisors studied in \cite{gian-gib}
from the point of view of GIT and conformal blocks; see Example \ref{E:typeA},
and \S\ref{S:examples-CB} for other connections to the conformal blocks divisors.

Questions about divisors on $\M_{0,n}$ tend to be much more difficult
than questions about divisors on $\M_{0,n}/\Sm_n$, or, equivalently,
questions about symmetric divisors on $\M_{0,n}$. In particular,
from the perspective of the F-conjecture, 
there are two downsides to working with the full Picard group
of $\M_{0,n}$: 
\begin{enumerate}
\item An absence of a natural basis for $\Pic(\M_{0,n})$.
\item The combinatorial complexity of the F-cone inside $\Pic(\M_{0,n})$.
\end{enumerate}
The first issue will be discussed in \S\ref{S:effective-criterion}. To illustrate the second issue,
we recall that on $\M_{0,5}$ there are $5$ extremal rays, and
on $\M_{0,6}$ there are $3,190$ extremal rays of the F-cone, 
and as far as we know the number of extremal rays of the F-cone 
on $\M_{0,7}$ is unknown. 

In contrast, the subgroup of symmetric divisors in $\Pic(\M_{0,n})$
has a natural basis given by the boundary divisors $\Delta_2, \Delta_3, \dots, \Delta_{\nhalf}$,
which incidentally generate the cone of effective divisors on $\M_{0,n}/\Sm_n$ \cite{KMc}.
The combinatorial complexity of the symmetric F-cone also grows more slowly
with $n$. For example, the symmetric
F-cone has only $1$, respectively $2$, extremal rays for $n=5$, respectively $n=6$. 
For $n=25$, the symmetric F-cone has $28,334$ extremal rays, but can still be computed
in under $3$ minutes on a personal computer. 

A major downside of working with symmetric divisors is that the restriction 
to a boundary of a symmetric divisor is no longer symmetric, whereas
many arguments concerning nefness of divisors on $\M_{0,n}$ are inductive and
make use of the restriction to the boundary. For example, 
one of the more involved parts of the proof of the symmetric F-conjecture for $n\leq 24$
involves verifying nefness of the \emph{necessary boundary restrictions} \cite[Definition 4.6]{Gib}
of F-nef symmetric divisors. 


A simple, but crucial idea behind this work is that we should consider
 families of divisors on $\M_{0,n}$ that are closed under the operation of
restriction to the boundary 
while that are not too far from symmetric divisors. 
Plainly, we want to consider all the symmetric divisors and all the pullbacks
of the symmetric divisors under various attaching morphisms.
To formalize this idea, we find it convenient to use the language of symmetric
functions on finite cyclic groups (or more generally on abelian groups) to 
describe the family of divisors we are working with. If $G$ is an abelian group
and $f\co G\to \ZZ$ is a symmetric function, then for every $n$-tuple $(d_1,\dots,d_n)$ of elements
of $G$,
we consider the divisor of the following form:
\begin{equation}\label{E:divisor-intro}
\DD \bigl(G, f; (d_1,\dots,d_n)\bigr):=
\sum_{i=1}^n f(d_i) \psi_i -\sum_{I,J} f\Bigl(\sum_{i\in I} d_i \Bigr)\Delta_{I,J}.
\end{equation}
To describe when the divisor \eqref{E:divisor-intro} is F-nef,
we introduce a notion of an F-nef function (Definition \ref{D:F-nef}), 
defined by a numerical condition reminiscent of the F-nefness condition on a symmetric divisor.
The correspondence between symmetric F-nef divisors and F-nef 
functions is not one-to-one. Instead, there are infinitely many F-nef
divisors that are associated to a single F-nef function. 
Importantly, every symmetric F-nef divisor does arise from some
F-nef function. 
The language of symmetric functions and
accompanying notation is introduced in Section \ref{S:divisors-from-functions}, 
where we also make explicit factorization properties of our divisors.   

An important outcome of the new viewpoint is that
every symmetric F-nef divisor can be put into an infinite family 
of (not necessarily symmetric) F-nef divisors associated to the same 
F-nef function. In Section \ref{S:nefness-bpf}, 
we introduce three effectivity notions for F-nef functions, namely 
effectivity, tree-effectivity, and cyclic effectivity. We show that we can prove semiampleness (respectively,
nefness) of all the infinitely many divisors associated to a single function by proving the function
to be tree-effective (respectively, effective); see Theorem \ref{P:semiample}. 

In Section \ref{S:cyclic}, we reinterpret the condition for a function $f\co \ZZ_m \to \QQ$
to be cyclically effective in terms of a balancedness condition on a certain 
quadratic form associated to $f$. As a result, we obtain an ``if and only if'' criterion 
for cyclic effectivity of functions on $\ZZ_m$ and an infinitude of semiample divisors. 
In Section \ref{S:2nd-semiampleness}, we prove a sufficient condition 
for a function to be tree-effective, thus obtaining another infinite collection of 
semiample divisors. 
In Section \ref{S:new-nef}, we explain how to obtain new effective functions from
certain cyclically effective functions. Thus we obtain new infinite families of nef divisors. 
Finally, in Section \ref{S:examples}, we collect some examples and 
computational results, and 
obtain a proof of the symmetric F-conjecture in arbitrary characteristic for $n\leq 16$.



\subsubsection*{Notation and conventions:} 
We use freely standard divisor theory of $\M_{0,n}$,
as in \cite{KMc} and \cite{GKM}, where the reader can also 
find an extended discussion of F-nefness for divisors on $\M_{0,n}$.
We let $[n]:=\{1,\dots, n\}$ and write $\Sm_n$ for the symmetric group on $[n]$.

\section{Preliminaries}
\label{S:preliminary}

\subsection{Boundary stratification of $\M_{0,n}$}
\label{S:boundary}
In this subsection, we recall, in the language that will be used throughout this work, 
well-known facts about the stratification of $\M_{0,n}$ by topological type;
see \cite[Section 1]{AC}, 
or more recent expository 
articles \cite[p.17]{ian-mori} or \cite[p.650]{vakil-taut} for more details.
\subsubsection*{Dual graphs}
Given $[C]\in \M_{0,n}$, \emph{the dual graph of $C$} is the tree graph $G(C)$ such that 
\begin{itemize}
\item Internal nodes of $G(C)$ correspond 
to the irreducible components of $C$.
\item Internal edges of $G(C)$ correspond to the nodes of $C$.
\item The leaves of $G(C)$ are labeled by $[n]$ and correspond to the marked points. 
\item Each internal node of $G(C)$ has valency at least $3$. 
\end{itemize}
We will refer to any tree on $n$ leaves satisfying the above conditions as \emph{$[n]$-labeled}. 

If $T$ is an $[n]$-labeled tree, then every edge $e$ 
of $T$ defines in a natural way 
the \emph{$e$-partition} $[n]=I \sqcup J$ such that $I$ and $J$ are the sets of leaves of $T$ lying 
on the two connected graph components resulting from deleting $e$ from $T$. 
The $e$-partitions, where $e$ ranges over the edges of a fixed $T$, will be called $T$-partitions.

With the above terminology, 
the locally closed boundary strata of $\M_{0,n}$ are in bijection with the set of all $[n]$-labeled trees:
The boundary stratum $M_T$ associated to an $[n]$-labeled tree $T$ consists of all $[C]\in \M_{0,n}$
such that $G(C) \simeq T$. For example, if $T_0$ is the $[n]$-labeled tree with one internal node, then $M_{T_0}=M_{0,n}$
is the interior of $\M_{0,n}$.

We note for the future that $[C]\in \Delta_{I,J}$ if and only if $I \sqcup J$ is a $G(C)$-partition. 
Consequently, we have that 
\begin{equation*}
\M_{T}=\bigcap_{\substack{I\sqcup J=[n] \text{ is a proper} \\ \text{  $T$-partition}}} \Delta_{I,J} \, .
\end{equation*}

\subsubsection*{Factorization of the restriction of a line bundle to the boundary}
Suppose $T$ is an $[n]$-labeled tree.
Let $V(T)$ be the set of internal nodes of $T$ and, for every $x\in V(T)$,
let $v(x)$ denote the valence of $x$. The closed boundary stratum $\M_T$ 
has the following product description:
\begin{equation}\label{E:boundary-product}
\M_T\simeq \prod_{x\in V(T)} \M_{0,v(x)}\, .
\end{equation}
If $L$ is a line bundle on $\M_{0,n}$, then 
the restriction of $L$ to $\M_T$
can be written as a product of pullbacks of line bundles on $\M_{0,v(x)}$. 
Namely,
\[
L\vert_{\M_T}=\bigotimes_{x\in V(T)} \operatorname{pr}_x^*L_x,
\]
where $\operatorname{pr}_x\co \M_T \to \M_{0,v(x)}$ is the projection and 
$L_x\in \Pic(\M_{0,v(x)})$;
this follows, for example, from \cite[Lemma 1.1]{GKM}. 
We will call $L_x$ \emph{the factors} of the restriction 
of $L$ to $\M_T$. 
Importantly, each factor $L_x$ can be realized in the following way.
Let $j_x \co \M_{0,v(x)} \to \M_{0,n}$ be (any) attaching morphism associated to the $v(x)$
edges emanating from $x$. Then $L_x=j_x^*(L)$. 

\subsubsection*{Zero and one-dimensional boundary strata}
\label{S:01}
Let $c$ be the number of the 
internal edges of an $[n]$-labeled tree $T$. Then the codimension of $\M_T$ in $\M_{0,n}$ 
equals $c$. In particular, the $0$-dimensional boundary strata of $\M_{0,n}$ are in bijection
with the set of $[n]$-labeled unrooted binary trees. We denote the set of $0$-dimensional strata of 
$\M_{0,n}$ by $B(n)$. It is well-known (see e.g., \cite[p.395]{sturmfels-speyer}) that \[\card(B(n))=(2n-5)!!\]
The $1$-dimensional closed boundary strata in $\M_{0,n}$ are called \emph{F-curves}. The numerical
classes of F-curves are in bijection with partitions 
\[
[n]=I\sqcup J\sqcup K\sqcup L
\]
of the set of marked points into $4$ non-empty subsets. The F-conjecture \ref{F-conjecture} 
is equivalent to the statement that the classes of F-curves generate the cone of curves of $\M_{0,n}$.

\subsection{Weighted graphs}
\label{S:weighted-graphs}
Our eventual goal is to arrive at combinatorial criteria for semiampleness of certain 
divisors on $\M_{0,n}$, which will be formally introduced in Section \ref{S:divisors-from-functions}. 
These criteria will be stated Section \ref{S:nefness-bpf} in the language of weightings on complete
graphs, which we now introduce. 

We denote by $\Gamma([n])$ the complete graph whose vertices are identified
with $[n]$ and by $E([n])$ the set of all edges
of $\Gamma([n])$.  We write $(i\sim j)$ to denote the edge joining vertices $i$ and $j$. 
A \emph{weighting} or a \emph{weight function} on $\Gamma([n])$ is a function $w\co E([n]) \ra \QQ$.
We write $w(i\sim j)$ to denote the weight of the edge $(i\sim j)$.
Given a weight function $w$ on $\Gamma([n])$,
we make the following definitions:
\begin{enumerate}
\item The \emph{$w$-flow through a vertex} $k\in [n]$ is defined to be 
\[
w(k):=\sum_{i\neq k} w(k\sim i).
\]
\item The \emph{$w$-flow across a partition} 
$I\sqcup J=[n]$ is defined to be 
\[
w(I\mid J):=\sum_{i\in I, j\in J} w(i\sim j).
\] 
\end{enumerate}

We will denote by $P_n$ the regular $n$-gon and, by abuse of notation, the set of its vertices. 
With this notation, a \emph{cyclic ordering}
of $[n]$ will be a bijection $\sigma \co [n] \to P_n$, modulo the rotations of $P_n$. 
We say that a partition $A\sqcup B=P_n$ is \emph{contiguous}, if $A$ (equivalently, $B$) 
are contiguous vertices in $P_n$.
For a cyclic ordering $\sigma$ of $[n]$, we say that a partition $I\sqcup J=[n]$
is \emph{$\sigma$-contiguous} if $\sigma(I)\sqcup \sigma(J)=P_n$ is a contiguous partition.

\subsection{Effective boundary criterion}
\label{S:effective-criterion}
The Picard group of $\M_{0,n}$ is well-understood (see \cite{keel}, \cite[\S2]{AC}, or \cite[\S1]{GKM}). 
Namely, $\Pic(\M_{0,n})$ is a torsion-free abelian group generated by 
the boundary divisors $\Delta_{I,J}$, where $I\sqcup J=[n]$ is a
proper partition, and the cotangent line bundles ($\psi$-classes) $\psi_k$, where $k=1,\dots, n$.
From \cite[Theorem 2.2(d)]{AC}, which in turn follows from Keel's relations \cite{keel},
the group of relations among the generators is generated by the following elements,
defined for every $i\neq j$:
\begin{equation}\label{E:relation}
\psi_i+\psi_j=\sum_{i\in I,\, j\in J}\Delta_{I,J}.
\end{equation} 
It follows that every line bundle on $\M_{0,n}$ can be written as 
\begin{align*}
\sum_{k=1}^n a_k \psi_k - \sum_{I,J} b_{I,J} \Delta_{I,J},
\end{align*}
but also that this representation is far from unique. 
\begin{remark}[Some bases of $\Pic(\M_{0,n})$]
\label{R:nonadjacent}
It is possible to write down several bases of $\Pic(\M_{0,n})$. For example,
 \[
 \bigl\{\psi_1,\dots, \psi_n\bigr\} \cup \bigl\{\Delta_{I,J} \mid \vert I\vert, \vert J \vert \geq 3\bigr\},
 \]
is such a basis by \cite[Lemma 2]{FG}. 

An important series of bases of $\Pic(\M_{0,n})$
is given by \emph{non-adjacent bases} discovered by Gibney and Keel. To describe 
a non-adjacent basis, 
fix a cyclic ordering $\sigma\co [n] \to P_n$. Then the set
\[
\mathcal{A}_{\sigma}:=\bigl\{\Delta_{I,J} \mid \text{$I\sqcup J=[n]$ is not a $\sigma$-contiguous partition}\bigr\}
\] 
is a basis of $\Pic(\M_{0,n})$ (see \cite{carr-2009}).
As we will see, the fact that $\mathcal{A}_{\sigma}$ is a basis of $\Pic(\M_{0,n})$ is equivalent
to the existence of a certain weighting on $[n]$. 
\end{remark}

We now state a simple observation that we will use repeatedly in the sequel.
\begin{lemma}[Effective Boundary Lemma]\label{L:main} Let $R=\ZZ$ or $R=\QQ$.
Consider
\[
D=\sum_{k=1}^n a_k \psi_k - \sum_{I,J} b_{I,J} \Delta_{I,J}.
\] Then
$D=\sum_{I,J} c_{I,J}\Delta_{I,J}$ in $\Pic(\M_{0,n})\otimes R$ 
if and only if there is an $R$-valued weighting on 
$\Gamma([n])$ such that the flow through each vertex $k$ is $a_k$ and the
flow across each proper partition $I\sqcup J=[n]$ is $b_{I,J}+c_{I,J}$.

In particular, $D$ is represented by an effective $R$-linear combination of the boundary divisors
on $\M_{0,n}$ if and only if there exists an $R$-valued weighting on $\Gamma([n])$ such
that the flow through each vertex $k$ is $a_k$ and the
flow across each proper partition $I\sqcup J=[n]$ is at least $b_{I,J}$.
\end{lemma}
\begin{proof}
Suppose that for each $i\neq j$ we use the relation \eqref{E:relation} $w(i\sim j)$ times to rewrite $D$ as 
$\sum_{I,J} c_{I,J} \Delta_{I,J}$.
Then in the free $R$-module generated by $\{\psi_k\}_{k=1}^n$ and $\{\Delta_{I,J}\}$ we have
\begin{multline*}
\sum_{I,J} c_{I,J} \Delta_{I,J}=D-\sum_{i\neq j} w(i\sim j)\left(\psi_i+\psi_j-\sum_{i\in I, \, j\in J}\Delta_{I,J}\right)
\\=\sum_{k=1}^n \bigl(a_k-w(k)\bigr) \psi_k -\sum_{I,J} \bigl(b_{I,J}-w(I\mid J)\bigr) \Delta_{I,J}.
\end{multline*}
The claim follows. 
\end{proof}

\section{Divisors on $\M_{0,n}$ from symmetric functions on abelian groups}
\label{S:divisors-from-functions}
\subsection{Divisors on $\M_{0,n}$ from functions on abelian groups}
Let $G$ be an abelian group. A \emph{$G$-$n$-tuple} $\vec{d}=(d_1,\dots,d_n)$
is an $n$-tuple of elements of $G$ satisfying $\sum_{i=1}^n d_i=0 \in G$.
Given $I\subset [n]$, we set $\vec{d}(I):=\sum_{i\in I} d_i$. 

A function $f\co G\to \QQ$ is \emph{symmetric} if 
$f(a)=f(-a)$ for every $a\in G$.  
Given a symmetric function $f$ on $G$
and a $G$-$n$-tuple $\vec{d}$, 
we define the following line bundle on $\M_{0,n}$:
\begin{equation}\label{E:divisor}
\DD \bigl(G, f; \vec{d}\bigr)=
\sum_{i=1}^n f(d_i) \psi_i -\sum_{\substack{I\sqcup J=[n]\\ \text{is proper}}} f\left(\vec{d}(I)\right)\Delta_{I,J}.
\end{equation}
The line bundles defined by the above formula will be called \emph{parasymmetric}
(this name will be motivated in the following discussion that culminates in Proposition \ref{P:parasymmetric}).


We begin by explaining why we are forced to consider the line bundles 
defined by \eqref{E:divisor}
even if we are ultimately interested in symmetric line bundles on $\M_{0,n}$. 
To begin, note that \eqref{E:divisor} describes all symmetric divisors on $\M_{0,n}$.
Indeed, the line bundles on $\M_{0,n}$ obtained by taking $G=\ZZ_m$, 
where $m\mid n$, and $d_1=\cdots=d_n=1$, are symmetric.
Conversely, any symmetric $\QQ$-divisor $D=\sum_{r=2}^{\nhalf} a_r \Delta_r$ on $\M_{0,n}$
can be written as $\DD \bigl(\ZZ_n, p_D; (\underbrace{1,\dots,1}_n)\bigr)$, by taking $p_D\co \ZZ_n \to \QQ$ defined by
\begin{equation}\label{E:pd-function}
\begin{aligned}
p_D(0)&=p_D(1)=0, \\ 
p_D(n-r)&=p_D(r)=-a_r, \text{ for $2\leq r\leq \nhalf$}.
\end{aligned}
\end{equation} 

Second, parasymmetric line bundles satisfy the following property:
\begin{lemma}[Functoriality with respect to the boundary stratification]\label{L:functorial} 
Fix a symmetric function $f\co G \ra \QQ$. 
Then for every $G$-$n$-tuple $\vec{d}=(d_1,\dots,d_n)$ and every 
boundary divisor $\Delta_{I,J} \subset \M_{0,n}$, we have 
\[
\DD\bigl(G, f; (d_1,\dots, d_n)\bigr)\vert_{\Delta_{I,J}}=\DD\bigl(G, f; 
(\{d_i\}_{i\in I}, \sum_{j\in J} d_j)\bigr)\boxtimes 
\DD\bigl(G, f; (\sum_{i\in I} d_i, \{d_j\}_{j\in J})\bigr),
\]
where we use the standard identification $\Delta_{I,J}\simeq \M_{0, \, I\cup p}\times \M_{0, \, q\cup J}$.

Consequently, for a fixed $G$ and $f\co G\to \QQ$, every factor of $\DD\bigl(G, f; (d_1,\dots, d_n)\bigr)$
on every closed boundary stratum of $\M_{0,n}$ is of the form $\DD\bigl(G, f; \vec{g}\bigr)$,
where $\vec{g}$ is some $G$-$k$-tuple. 
\end{lemma}

\begin{proof} Follows from standard intersection theory on $\M_{0,n}$.
Indeed, recall that for any $A \subsetneq I$ and $B=[n]- A$, we have
\begin{equation}\label{E1}
\cO_{\M_{0,n}}(\Delta_{A,[n]-A})\vert_{\Delta_{I,J}}=\cO_{\M_{0, I\cup p}}(\Delta_{A,I\cup p-A}) \boxtimes \cO_{\M_{0, J\cup q}}
\end{equation}
and for $B\subsetneq J$, we have 
\begin{equation}\label{E2}
\cO_{\M_{0,n}}(\Delta_{B,[n]-B})\vert_{\Delta_{I,J}}=\cO_{\M_{0, I\cup p}}
\boxtimes \cO_{\M_{0, J\cup q}}(\Delta_{B,J\cup q-B}).
\end{equation}
More importantly, 
\begin{equation}\label{E3}
\cO_{\M_{0,n}}(-\Delta_{I,J})\vert_{\Delta_{I,J}}=\psi_p
\boxtimes \psi_q 
\end{equation}
on $\Delta_{I,J}\simeq \M_{0, I\cup p}\times \M_{0, J\cup q}$.

All other boundary divisors restrict to $0$ on $\Delta_{I,J}$, and the
$\psi$-classes on $\M_{0,n}$ restrict to the corresponding $\psi$-classes on 
$\Delta_{I,J}\simeq \M_{0, I\cup p}\times \M_{0, J\cup q}$.
We conclude that by summing \eqref{E1} with the coefficient 
$-f\left(\sum_{i\in A} d_i\right)$, \eqref{E2} with the coefficient 
$-f\left(\sum_{j\in B} d_j\right)$, \eqref{E3} with the coefficient
$f\left(\sum_{i\in I} d_i\right)=f\left(\sum_{j\in J} d_j\right)$, 
and each $\psi_i$ with the coefficient $f(d_i)$, 
we obtain 
precisely
\[
\DD\bigl(G, f; (d_1,\dots, d_n))\vert_{\Delta_{I,J}}=\DD\bigl(G, f; 
(\{d_i\}_{i\in I}, \sum_{j\in J} d_j)\bigr)\boxtimes 
\DD\bigl(G, f; (\{d_j\}_{j\in J}, \sum_{i\in I} d_i)\bigr).
\]
The second part follows from the first using the discussion of \S\ref{S:boundary}.
\end{proof}

Third, if $G$ is a finite cyclic group, then every parasymmetric divisor associated to $G$
is a pullback of some symmetric divisor on a higher-dimensional moduli space 
under some attaching morphism:
\begin{lemma}\label{L:pullback-of-symmetric} Suppose $G=\ZZ_m$ and $\{d_i\}_{i=1}^n$ 
are positive integers
such that $m\mid N:=\sum_{i=1}^n d_i$. 
Consider an attaching morphism $j\co \M_{0,n} \to \M_{0, N}$ defined by attaching to the $i^{th}$
marked point a fixed element of $\M_{0, d_i+1}$. 
Then 
\[
\DD \bigl(G, f; (d_1,\dots, d_n)\bigr)=j^*\Big(\DD\bigl(G, f; (\underbrace{1,\dots, 1}_{N})\bigr)\Big).
\]
\end{lemma}
\begin{proof}
This follows immediately from Lemma \ref{L:functorial} and the discussion of \S\ref{S:boundary}.
\end{proof}

Motivated by the above two lemmas, we consider the set of all parasymmetric line bundles
in $\Pic(\M_{0,n})$
associated to an abelian group $G$, a symmetric function $f\co G\to \ZZ$, and 
an integer $n\geq 3$: 
\begin{equation}\label{E:sgfn}
\PS(G,f,n):=\left\{
\DD \bigl(G, f; \vec{d}\bigr)\mid \text{ 
$\vec{d}$ is a $G$-$n$-tuple
} \right\}.\end{equation}
We also set
\begin{equation}\label{E:sgn}
 \PS(G,n):=\bigcup_{\substack{f\co G\to \ZZ \\ \text{is symmetric}}} \PS(G,f,n).
 \end{equation}
Since every symmetric line bundle on $\M_{0,n}$ lies in $\PS(\ZZ_n, f, n)$
for some symmetric function $f\co \ZZ_n \to \ZZ$ and since
every line bundle in $\PS(\ZZ_m, f, n)$ is a pullback of a symmetric line bundle on some $\M_{0,N}$
via an attaching morphism by Lemma \ref{L:pullback-of-symmetric}, we obtain the following result:
\begin{prop}\label{P:parasymmetric} 
$\bigcup_{m\geq 2, n\geq 3} \PS(\ZZ_m, n)$ 
is the smallest collection of line bundles on $\{\M_{0,n}\}_{n\geq 3}$
that include all symmetric line bundles and is closed under pullbacks via attaching morphisms.
\end{prop}

From now on, we will be interested exclusively in nefness and semiampleness of parasymmetric divisors. 
In order to recast the F-conjecture for divisors in $\PS(G,n)$ in terms of symmetric functions
on the abelian group $G$, we make the following definition:
\begin{definition}\label{D:F-nef} 
Let $G$ be an abelian group. 
We say that a symmetric function $f\co G \ra \QQ$ is \emph{F-nef} if for any $a,b,c\in G$ we have
\begin{equation}\label{E:F-nef}
f(a)+f(b)+f(c)+f(a+b+c)\geq f(a+b)+f(a+c)+f(b+c).
\end{equation}
\end{definition}
\subsubsection*{First properties}
For an arbitrary $f\co G \ra \QQ$, we define $d_f \co G\times G \ra \QQ$ by 
$
d_f(a,b)=f(a)+f(b)-f(a+b).
$
Then \eqref{E:F-nef} is equivalent to subadditivity of $d_f$ in each variable:
\begin{equation}\label{E:diff-nef}
d_f(a,b)+d_f(a,c) \geq d_f(a,b+c).
\end{equation}
\begin{lemma}\label{L:subadditivity}
Suppose $f\colon G \ra \QQ$ is an F-nef function on a finite abelian group. Then $d_f$ is non-negative, and hence $f$ is subadditive. Namely,
for any $a,b\in G$, we have 
\[
f(a)+f(b)\geq f(a+b).
\]
\end{lemma}
\begin{proof}
Take a positive integer $N$ so that $Nb=0$. Applying \eqref{E:diff-nef}, we obtain
\[
(N+1)d_f(a,b)\geq d_f(a,(N+1)b)=d_f(a,b). 
\]
\end{proof}

\subsubsection*{Some examples of F-nef functions}
In the following examples, we take $G=\ZZ_m$.
For $k\in \ZZ$, we denote by $\modp{k}{m}$ the residue of $k$ in $\{0,1,\dots, m-1\}$. 
\begin{example} \label{E:standard}
We define the \emph{standard function} $A_m\co \ZZ_m \ra \ZZ$ by
\begin{equation*}
A_{m}(i)=\modp{i}{m}\modp{m-i}{m}.
\end{equation*}
Given $a,b,c \in \ZZ_m$, set $d=\langle m-a-b-c\rangle_m$. It is easy to check that 
 \begin{align}\label{E:standard-function}
A_m(a)+A_m(b)&+A_m(c)+A_m(a+b+c)- A_m(a+b)-A_m(a+c)-A_m(b+c) \notag \\ &=\begin{cases} 0 \qquad \text{if $a+b+c+d=m$ \ \ or \ \ $a+b+c+d=3m$,}  \\
 2m \cdot \min\{ a, b, c, d, m-a, m-b, m-c, m-d \}, \ \text{otherwise.}
 \end{cases}
\end{align}
We conclude that $A_m$ is F-nef. 
\end{example}

\begin{example}\label{example-B}
The \emph{$2^{nd}$ standard function} $B_m\co \ZZ_m \ra \ZZ$ is defined by 
\[
\text{$B_m(i)=A_m(i)$ if $i\neq 1, m-1$, \ \ and \ $B_m(1)=B_m(m-1)=3m-1$.}\]
It is easy to verify using \eqref{E:standard-function} that $B_m$ is F-nef for all $m\geq 8$ and for $m=4,6$,
but is not F-nef for $m=5,7$.
\end{example}

\begin{example}\label{example-E}
The \emph{$3^{rd}$ standard function} $E_m\co \ZZ_m \ra \ZZ$ is defined by 
\[
\text{$E_m(i)=A_m(i)$ if $i\neq 0$, \ \ and \ $E_m(0)=m$.}\]
It is easy to verify using \eqref{E:standard-function} that $E_m$ is F-nef for all $m\geq 3$, but is not F-nef for $m=2$.
\end{example}

\subsection{F-nef divisors from F-nef functions}
Define the following subset of $\PS(G,n)$ from \eqref{E:sgn}:
\begin{equation}
\PSFN(G,n):=\left\{
\DD \bigl(G, f; \vec{d}\bigr)\mid \text{
$f\co G \to \ZZ$ is an F-nef function, $\vec{d}$ is a $G$-$n$-tuple} \right\}.
\end{equation}

We begin by observing that every divisor in $\PSFN(G,n)$ is F-nef on $\M_{0,n}$:
\begin{lemma}\label{L:F-nef} 
Suppose $f\co G \ra \QQ$ is F-nef.
Then the divisor $\DD(G, f; (d_1,\dots, d_n))$ is F-nef on $\M_{0,n}$.
\end{lemma}
\begin{proof} 
Consider a partition $I \sqcup J \sqcup K \sqcup L=[n]$ and the corresponding F-curve $F_{I,J,K,L}$.
Set $A=\sum_{i\in I} d_i$, $B=\sum_{j\in J} d_j$, $C=\sum_{k\in K} d_k$, $D=\sum_{\ell\in L} d_\ell$.
Note that $A+B+C+D=0 \in G$ and so $f(D)=f(A+B+C)$.
Then using Lemma \ref{L:functorial}, we obtain
\begin{multline*}
\DD(G, f; (d_1,\dots, d_n))\cdot F_{I,J,K,L}
\\ =f(A)+f(B)+f(C)+f(D)- f(A+B)-f(A+C)-f(B+C) \geq 0,
\end{multline*}
by F-nefness of $f$.
\end{proof}
As already observed, every symmetric divisor $D$ on $\M_{0,n}$ comes from a 
symmetric function on $\ZZ_n$.
More importantly, 
every symmetric F-nef divisor on $\M_{0,n}$ comes from some F-nef function on $\ZZ_n$:
\begin{lemma}\label{L:every-F-nef}
Suppose $D=\sum_{r=2}^{\lfloor n/2\rfloor} a_r\Delta_r$ is a symmetric F-nef divisor on $\M_{0,n}$.
Then there exists an F-nef function $f_D\co \ZZ_n \to \QQ$ such that 
\begin{equation}\label{E:divisor-function}
D=\DD\bigl(\ZZ_n, f_D; (1)_n\bigr).
\end{equation}
\end{lemma}

\begin{proof}
Let $p_D\co \ZZ_n \ra \ZZ$ be as defined in \eqref{E:pd-function}.
Consider 
\[
\lambda_{\fnef}(D):=\max \left\{ \frac{p_D(a+b)+p_D(a+c)+p_D(b+c)-p_D(a)-p_D(b)-p_D(c)-p_D(d)}{2n\cdot \min\{ a, b, c, d, n-a, n-b, n-c, n-d \}}\right\},
\]
where the maximum is taken over $a,b,c,d\in \{1,\dots, n-1\}$ with $a+b+c+d=2n$.
Let $A_{n}$ be the standard function on $\ZZ_n$ as defined in Example \ref{E:standard}.
Then using \eqref{E:standard-function}, we see that $\lambda A_{n}+p_D$ is F-nef if and only if $\lambda\geq \lambda_{\fnef}(D)$. 
We define 
\begin{equation}\label{E:function-from-divisor}
f_{D}:=\lambda_{\fnef}(D) A_{n}+p_D.
\end{equation}

Observing that $\DD\bigl(\ZZ_n, A_n; (1)_n\bigr)=0$ (again from \eqref{E:standard-function}), 
we see that  
\begin{equation*}
\DD\bigl(\ZZ_n, f_D; (1)_n\bigr)
=\DD\bigl(\ZZ_n, p_D; (1)_n\bigr)=D.
\end{equation*}
We are done. 
\end{proof}

\begin{corollary}
$\{\PSFN(\ZZ_m,n)\}_{m\geq 2, n\geq 3}$ is the smallest collection of line bundles on 
$\{\M_{0,n}\}_{n\geq 3}$
that include all symmetric F-nef line bundles and is closed under pullbacks via attaching morphisms.
\end{corollary}

\begin{definition} The function $f_{D}$ constructed in the proof of Lemma
\ref{L:every-F-nef} (see \eqref{E:function-from-divisor}) will be called
the \emph{associated F-nef function} of the divisor $D$. 
\end{definition}

\section{Nefness and base point freeness}
\label{S:nefness-bpf}
\subsection{Effective boundary divisors of three kinds}
A divisor on $\M_{0,n}$ of the form $\sum b_{I,J}\Delta_{I,J}$, where $b_{I,J}\geq 0$, is called an
\emph{effective boundary}. We say that a line bundle $L \in \Pic(\M_{0,n})$ is an effective boundary 
if the linear system $\vert L\vert$ contains an effective boundary divisor. 

\begin{definition} We say that $L \in \Pic(\M_{0,n})$ is a \emph{stratally effective boundary}
if for every closed boundary stratum $\M_T \subset \M_{0,n}$, every factor of the restriction 
of $L$ to $\M_T$ (as defined in \S\ref{S:boundary}) is an effective boundary. 
\end{definition}

Recall that, by a standard argument called \emph{Effective Dichotomy} \cite[p.39]{ian-mori}, 
a stratally effective boundary divisor is nef. 
Indeed, suppose $L$ is a stratally effective boundary on $\M_{0,n}$.
Then $L$ intersects non-negatively all irreducible curves not lying entirely in 
$\Delta=\M_{0,n}\setminus M_{0,n}$, simply by virtue of being linearly equivalent to an effective divisor 
with the 
support in $\Delta$. Furthermore, the restriction of $L$ to every irreducible boundary
divisor in $\Delta$ is nef by the same argument using the assumption that $L$ is stratally effective boundary.
Hence $L$ is nef on $\M_{0,n}$. 

Historically, proving that every F-nef line bundle
on $\M_{0,n}$ is stratally effective boundary was one of the original approaches to the F-conjecture,
see for example \cite[Question 3.31, p.100]{ian-mori} or \cite[Question 0.13]{GKM}.
Pixton's example \cite{pixton} of a semiample line bundle on $\M_{0,12}$
that is not effective boundary shows that this approach
was too optimistic. One motivation behind this work is a (perhaps just as overly optimistic)
hope that every F-nef parasymmetric line bundle might still be stratally effective boundary. 

\begin{definition} We say that a line bundle $L \in \Pic(\M_{0,n})$ is \emph{boundary semiample}
if for every $[C]\in \M_{0,n}$, there exists an effective boundary $\QQ$-divisor $D_C$
linearly equivalent to $L$ such that $[C]\notin \Supp(D_C)$.
\end{definition}
Clearly, a boundary semiample line bundle is both semiample and stratally effective boundary. 
With this in mind, we make the following simple observation:
\begin{lemma}\label{L:semiample}
Suppose $L$ is a line bundle on $\M_{0,n}$ such that for every 
$0$-dimensional boundary stratum $B\in B(n)$, there exists an effective boundary $D_B \in \vert L\vert$
such that $B \notin \Supp(D_B)$. Then $\langle D_B \mid B\in B(n)\rangle$ is a base point free linear
subsystem of $\vert L\vert$
and $L$ is a boundary semiample line bundle.
\end{lemma}
\begin{proof}
This is obvious once we recall that every boundary stratum in $\M_{0,n}$ contains 
some $0$-dimensional boundary stratum in its closure: Namely, for every $[C]\in \M_{0,n}$,
consider a locally closed boundary stratum $M_T$ such that $[C]\in M_T$. Let
$B\in \M_T$ be a $0$-dimensional boundary stratum and $D_B \in \vert L\vert$
be an effective boundary such that $B\notin \Supp(D_B)$, then $[C]\notin \Supp(D_B)$ as well.
The claim follows. 
\end{proof}

\subsection{Effective functions of three kinds}
\label{S:effective-3-kinds}
The goal of this section is to use the Effective Boundary Lemma \ref{L:main} to reinterpret in terms of weightings on
$\Gamma([n])$ (see \S\ref{S:weighted-graphs}) what it means 
for parasymmetric divisors $\DD \bigl(G, f; \vec{d}\bigr)$ defined by \eqref{E:divisor} to be 
\begin{enumerate}
\item Effective boundary.
\item Stratally effective boundary.
\item Boundary semiample.
\end{enumerate}
We keep the terminology and notation of Section \ref{S:preliminary}. 
\begin{definition}\label{D:effective-with-respect}
Let $f\co G \to \QQ$ be a symmetric function and $\vec{d}=(d_1,\dots, d_n)$
be a $G$-$n$-tuple.
We introduce the following notions of \emph{effectivity of $f$ with respect to $\vec{d}$},
listed in the order of increasing strength:
\begin{enumerate}
\item 
We say that $f$ is \emph{effective with respect to $\vec{d}$} if there
exists a weighting $w\co E([n])\to \QQ$ such that 
\begin{enumerate}
\item $w(i)=f(d_i)$ for all $i\in [n]$.
\item $w(I\mid J)\geq f(\vec{d}(I))$
for all proper partitions $I\sqcup J=[n]$.
\end{enumerate}

\item Let $T$ be an $[n]$-labeled tree. 
We say that $f$ is \emph{$T$-effective with respect to $\vec{d}$} if
there exists a weighting $w\co E([n])\to \QQ$ such that 
\begin{enumerate}
\item $w(i)=f(d_i)$ for all $i\in [n]$.
\item $w(I\mid J)\geq f(\vec{d}(I))$
for all proper partitions $I\sqcup J=[n]$.
\item $w(I\mid J)=f(\vec{d}(I))$
for all $T$-partitions $I\sqcup J=[n]$.
\end{enumerate}

\item We say that $f$ is \emph{cyclically effective with respect to $\vec{d}$} if for every 
cyclic ordering $\sigma\co [n]\to P_n$ there exists a weighting $w\co E([n])\to \QQ$ such that 
\begin{enumerate}
\item $w(i)=f(d_i)$ for all $i\in [n]$.
\item $w(I\mid J)\geq f(\vec{d}(I))$
for all proper partitions $I\sqcup J=[n]$.
\item $w(I\mid J)=f(\vec{d}(I))$
for all $\sigma$-contiguous partitions $I\sqcup J=[n]$.
\end{enumerate}
\end{enumerate}
\end{definition}

\begin{remark}[Effectivity and F-nefness]\label{R:n=34}
Every symmetric function is cyclically effective with respect to every $G$-$3$-tuple 
$\vec{d}=(d_1,d_2,d_3)$. This is easily seen by taking 
\[
w(i \sim j)=\frac{f(d_i)+f(d_j)-f(d_i+d_j)}{2}.
\]
It is also easy to see that effectivity with respect to a $G$-$4$-tuple $\vec{d}$
is equivalent to cyclic effectivity with respect to $\vec{d}$. In addition, 
$f$ is effective with respect to every 
$G$-$4$-tuple if and only if $f$ is F-nef (cf. Definition \ref{D:F-nef}). 
\end{remark}

\begin{lemma}\label{L:implication-1}
We have the following implications for a symmetric function $f\co G\to \QQ$
with respect to a $G$-$n$-tuple $\vec{d}$:
\[
\text{cyclically effective} \Longrightarrow \text{$T$-effective for all $[n]$-labeled trees $T$}
\Longrightarrow \text{effective}. 
\]
\end{lemma}
\begin{proof}
The second implication is obvious. 
For the first implication, observe that any $[n]$-labeled tree $T$ can be embedded into the 
plane so that the leaves
are identified with the vertices of $P_n$.
The resulting cyclic ordering $\sigma\co [n] \to P_n$ 
satisfies the property that every $T$-partition $I\sqcup J$ of $[n]$ is $\sigma$-contiguous.  
\end{proof}

\begin{lemma}\label{L:T-effective} Let $f\co G\to \QQ$ be a symmetric function and 
$\vec{d}=(d_1,\dots, d_n)$ be a $G$-$n$-tuple. Consider
the divisor $\DD \bigl(G, f; \vec{d} \bigr)$ from \eqref{E:divisor}.
\begin{enumerate}
\item If $f$ is effective with respect to $\vec{d}$, then  $\DD \bigl(G, f; \vec{d} \bigr)$
is an effective boundary.
\item Let $T$ be an $[n]$-labeled tree.
Suppose $f$ is $T$-effective with respect to $\vec{d}$.
Then there is an effective boundary $\QQ$-divisor $D$ on $\M_{0,n}$ such that 
$D\sim_{\QQ} \DD \bigl(G, f; \vec{d} \bigr)$ and $\Supp(D)$ does not
contain the generic point of $M_T \subset \M_{0,n}$. 
\end{enumerate}
\end{lemma}

\begin{proof}
This follows immediately from the Effective Boundary Lemma \ref{L:main} and
the description of $M_T$ from \S\ref{S:boundary}.
\end{proof}

The previous lemma motivates the following definitions:

\begin{definition}\label{D:effective} 
We say that a symmetric function $f\co G \to \QQ$ is \emph{effective}
if, for all $n\geq 3$, $f$ is effective with respect to every $G$-$n$-tuple $\vec{d}$.
\end{definition}

\begin{definition}\label{D:strongly-semiample} 
We say that a symmetric function $f\co G \to \QQ$ is \emph{tree-effective}
if, for all $n\geq 3$, and for every $[n]$-labeled tree $T$, the function $f$ is $T$-effective 
with respect to every $G$-$n$-tuple $\vec{d}$.
\end{definition} 
\begin{remark}
For tree-effectivity, it suffices to consider only $[n]$-labeled unrooted binary  trees, because
every other $[n]$-labeled tree is obtained from some unrooted binary tree by a series
of edge contractions. 
\end{remark}

\begin{definition}\label{D:cyclically-semiample} 
We say that a symmetric function $f\co G \to \QQ$ is \emph{cyclically effective}
if, for all $n\geq 3$, $f$ is cyclically effective with respect to every $G$-$n$-tuple $\vec{d}$. 
\end{definition} 

\begin{remark}\label{R:implications}
It follows from  Remark \ref{R:n=34} and Lemma \ref{L:implication-1}
that for a symmetric function $f\co G\to \QQ$, we have the following implications
\[
\text{cyclically effective} \Longrightarrow \text{tree-effective}
\Longrightarrow \text{effective} \Longrightarrow \text{F-nef}.
\]
\end{remark}

\subsection{Semiampleness and nefness of divisors from effectivity of functions}

A single tree-effective (resp., effective) 
function gives rise to an infinitude of semiample (resp., nef) divisors,
obtained by increasing $n$ and varying the choice of a $G$-$n$-tuple:

\begin{theorem}\label{P:semiample}
 Suppose $f\co G \to \QQ$ is tree-effective (respectively, effective).
Then for any $G$-$n$-tuple $\vec{d}=(d_1,\dots, d_n)$, the line bundle 
$\DD \bigl(G, f; (d_1,\dots, d_n)\bigr)$ is boundary semiample 
(respectively, stratally effective boundary, hence nef) 
on $\M_{0,n}$. 
\end{theorem}
\begin{proof}
Suppose $f\co G \to \QQ$ is tree-effective. Let $\vec{d}=(d_1,\dots, d_n)$ be a $G$-$n$-tuple.
Consider a $0$-dimensional boundary stratum $B$ of $\M_{0,n}$ and let $T$  
be the associated  $[n]$-labeled unrooted binary tree. By assumption, $f$ is $T$-effective with respect
to $\vec{d}$. It follows by Lemma \ref{L:T-effective}(2) 
that there is a $\QQ$-divisor $D$ on $\M_{0,n}$ such that 
$D\sim_{\QQ} \DD \bigl(G, f; \vec{d} \bigr)$ and $\Supp(D)$ does not contain
$B$. We conclude that $\DD \bigl(G, f; \vec{d}\bigr)$ is boundary semiample
by Lemma \ref{L:semiample}.

Suppose $f\co G \to \QQ$ is effective. Let $\vec{d}=(d_1,\dots, d_n)$ be some $G$-$n$-tuple and
set $L:=\DD \bigl(G, f; (d_1,\dots, d_n)\bigr)$. 
Then by Lemma \ref{L:functorial}, every factor of $L$ on every closed boundary stratum of $\M_{0,n}$
is of the form $\DD \bigl(G, f; \vec{g}\bigr)$ for some $G$-$k$-tuple $\vec{g}$. Since $f$ is effective
with respect to every such $\vec{g}$ by assumption, we conclude that
$\DD \bigl(G, f; \vec{g}\bigr)$ is effective boundary by Lemma \ref{L:T-effective}(1). 
\end{proof}

While Theorem \ref{P:semiample} allows one to prove semiampleness (resp., nefness)
of an infinite collection of line bundles at once by proving tree-effectivity (resp., effectivity)
of a single symmetric function, we are often interested in proving semiampleness
(resp., nefness) of a particular divisor on $\M_{0,n}$ without having to prove a more general
result about the associated F-nef function. In what follows, we introduce 
weakened versions of effectivity for functions on finite cyclic groups that allow us to do precisely that. 

\begin{definition}\label{D:weakly-effective} 
We say that a symmetric function $f\co \ZZ_n \to \QQ$ is \emph{weakly effective}
if $f$ is effective with respect to 
every $\ZZ_n$-$k$-tuple $\vec{d}$, where 
$d_1,\dots,d_k$ are non-negative integers satisfying $d_1+\cdots+d_k=n$.
\end{definition} 

\begin{definition}\label{D:weakly-semiample} 
We say that a symmetric function $f\co \ZZ_n \to \QQ$ is \emph{weakly tree-effective}
if $f$ is $T$-effective with respect to every $\ZZ_n$-$k$-tuple $\vec{d}$
for every $[k]$-labeled tree $T$, where 
$d_1,\dots,d_k$ are non-negative integers satisfying $d_1+\cdots+d_k=n$.
\end{definition} 

\begin{definition}\label{D:weakly-cyclically-semiample} 
We say that a symmetric function $f\co \ZZ_n \to \QQ$ is \emph{weakly cyclically effective}
if $f$ is cyclically effective with respect to every $\ZZ_n$-$k$-tuple $\vec{d}$, where 
$d_1,\dots,d_k$ are non-negative integers satisfying $d_1+\cdots+d_k=n$.
\end{definition}


\begin{remark}\label{R:weak-implications}
For a symmetric function $f\co \ZZ_n \to \QQ$, we have 
\[
\text{weakly cyclically effective} \Longrightarrow \text{weakly tree-effective}
\Longrightarrow \text{weakly effective}.
\]
\end{remark}

Repeating the proof of Theorem \ref{P:semiample} with minor modifications
we obtain the following result:
\begin{theorem}\label{P:weakly-semiample} 
Suppose $f\co \ZZ_n \to \QQ$ is weakly tree-effective (respectively, weakly effective).
Then for any $k$-tuple $\vec{d}=(d_1,\dots, d_k)$ of non-negative integers satisfying $\sum_{i=1}^k
d_i=n$, the line bundle 
$\DD \bigl(\ZZ_n, f; (d_1,\dots, d_k)\bigr)$ is boundary semiample 
(respectively, stratally effective boundary, hence nef)
 on $\M_{0,k}$.  
\end{theorem}

\section{Cyclic semiample divisors}
\label{S:cyclic}

In this section, we reinterpret the (weak) cyclic effectivity of F-nef
functions on $\ZZ_m$ in terms of certain properties for cyclic quadratic forms in $m$ variables, namely
\emph{balancedness} and \emph{weak balancedness} (see Definition \ref{D:balanced}
and Proposition \ref{P:cyclic=balanced}).

For a positive definite quadratic form, the balancedness conditions can be verified algorithmically. To a
balanced quadratic form, we can associate an infinite series of semiample divisors 
on various $\M_{0,n}$.
Conversely, to a symmetric divisor $D$ on $\M_{0,n}$, we associate 
a cyclic quadratic form $Q_D$ in $n$ variables, whose coefficients
depend linearly on the coefficients of $D$. We prove that $D$ is base point free if $Q_D$ is weakly balanced.
Weak balancedness of $Q_D$ is equivalent to 
finitely many linear inequalities on the coefficients of $D$ and so can be verified by a direct computation.
This gives an effective sufficient criterion for semiampleness (see Theorem \ref{T:semiample}).

\subsubsection*{Notation}
Throughout this section, we always have $G=\ZZ_m$ for some positive integer $m$.
As before, for $k\in \ZZ$, we denote by $\modp{k}{m}$ the residue of $k$ in $\{0,1,\dots, m-1\}$. 
All quadratic forms have rational coefficients, unless specified otherwise.
We use the shorthand $(a)_{r}$ for $\underbrace{a,\dots, a}_r$. 

\subsection{Balanced cyclic quadratic forms}
\label{S:quadratic}
A quadratic form $Q(x_0,\dots,x_{m-1})=\sum\limits_{0\leq i, \, j \leq m-1} Q_{i,j}x_ix_j$ is \emph{cyclic}
if \[Q(x_0,x_1,\dots,x_{m-1})=Q(x_1,\dots,x_{m-1},x_0).\]
Equivalently, $Q$ is cyclic if and only if its symmetric matrix is circulant, that is 
\[Q_{i,j}=q(i-j),\quad \text{where
$q\co \ZZ_m \ra \QQ$ is a symmetric function.}
\]

We now introduce a series of conditions on cyclic quadratic forms in $m$ variables:
\begin{condition}[$n$]\label{cond} For an integer $n=mc+r$, where $r=\modp{n}{m}$,
the minimum value of $Q$ at the integral points of the affine hyperplane 
\[
\sum_{i=0}^{m-1} x_i = n
\]
is achieved at the vector\footnote{We allow the minimum value 
to be achieved at other points; in particular, $Q$ attains the same value at all cyclic shifts of 
$v_{n}$.} 
$v_{n}:=\bigl((1+c)_{r}, (c)_{m-r}\bigr)$.
\end{condition}
\begin{example}
$x_0^2+x_1^2+\cdots+x_{m-1}^2$ satisfies Condition $(n)$ for all $n\in \ZZ$.
\end{example}

\begin{definition}[Balanced quadratic forms]\label{D:balanced} Let $Q$ be a cyclic quadratic form in $m$ variables.
\begin{enumerate}
\item[(a)]
We say that $Q$ is \emph{balanced} if Condition ($r$) holds for every $r \pmod{m}$.
\item[(b)]
We say that $Q$ is \emph{weakly balanced} if when restricted to $x_i\in \{0,1\}$ Condition $(r)$ holds for all $r=1,\dots, m-1$.

Equivalently, $Q$ is weakly balanced if and only if for $r=1,\dots, m-1$,
the leading principal $r\times r$ minor of $Q$ has the minimal 
sum of entries among all principal $r\times r$ minors.
\item[(c)]
We say that $Q$ is \emph{$\ell$-balanced}  if when restricted to $0 \leq x_i \leq \ell$ Condition $(r)$ holds for all $r \pmod{m}$. 
Thus, $1$-balanced forms are weakly balanced and $\infty$-balanced forms are balanced. 
\end{enumerate}
\end{definition}

We make three simple observations:

\begin{remark}[Finiteness of conditions] \label{R:finite-conditions}
For $Q=\sum\limits_{0\leq i, \, j \leq m-1} q(i-j)x_ix_j$, we have
\begin{multline}\label{E:shift}
Q(x_0+c,\dots, x_{m-1}+c)-Q(x_0,\dots,x_{m-1}) \\ =2c\left(\sum_{k=0}^{m-1} q(k)\right)(x_0+\dots+x_{m-1})+Q(c,\dots, c).
\end{multline}
It follows that Condition $(r)$ is satisfied if and only if Condition $(n)$ is satisfied for all $n \equiv \pm \, r \pmod{m}$. Moreover, to show that $Q$ is balanced, it suffices to check 
Condition $(r)$ for every $r\gg 0$ but restricted to positive $x_i$'s only.   
\end{remark}
\begin{remark}[Reduction to PSD] \label{R:PSD}
For each $k=0,\dots, m-1$, 
Condition ($0$) is equivalent to positive semi-definiteness of the quadratic form 
\[
\widetilde Q_k:=Q(x_0, \dots, x_{k-1},-\sum_{i\neq k} x_i, x_{k+1}, \dots, x_{m-1}).
\] 
Furthermore, note that $\widetilde{Q}:=\sum_{k=0}^{m-1} \widetilde Q_k$ satisfies 
\begin{equation}\label{E:PSD}
\widetilde{Q}=mQ+\bigl(mq(0)-2\sum_{k=0}^{m-1}q(k)\bigr)\bigl(x_0+\cdots+x_{m-1}\bigr)^2.
\end{equation}
We conclude that $Q$ satisfies Condition ($0$) if and only if $\widetilde{Q}$ is positive semi-definite,
and $Q$ satisfies Condition ($n$) if and only if $\widetilde{Q}$ does.
\end{remark}

\begin{remark}[Testing for balancedness] Suppose $Q$ is a positive definite cyclic quadratic form. Then balancedness of $Q$
is algorithmically verified in two steps by: \par
(1) Finding all vectors $v$ such that $Q(v) \leq \max \{ Q(v_r)  \mid r=0,\dots, m-1\}$.  \par
(2) Testing whether the vectors found in Step (1) violate Conditions ($n$).

For an arbitrary quadratic form $Q$ in $m$ variables, one can determine  
whether $Q$ is $\ell$-balanced by a direct evaluation of $Q$ at all integer points of the hypercube $[0,\ell]^m$. 
\end{remark}

We give several examples of balanced quadratic forms.
\begin{example}
$A(x_0,\dots, x_{m-1}):=\sum_{i=0}^{m-1} x_i^2$.
\end{example}

\begin{example}
$B(x_0,\dots, x_{m-1}):=\sum_{i=0}^{m-1}(x_i-x_{i+1}+x_{i+2})^2$, where $m\geq 4$ and $m \not\equiv \pm 1 \pmod{6}$.
\end{example}

\begin{example}
$C(x_0,\dots, x_{m-1}):=\sum_{i=0}^{m-1}(x_i+x_{i+k-1})^2$, where $m=2k$ and $k$ is odd.
\end{example}

\begin{example}
$D(x_0,\dots, x_{m-1}):=\sum_{i=0}^{m-1}(x_i+x_{i+k})^2$, where $m=2k+1$. 
\end{example}

Note that $A$ is obviously balanced. We prove balancedness of $B$ in the following lemma, and leave balancedness
of $C$ and $D$ as an exercise to the reader. 

\begin{lemma}\label{quadratic-B}
The quadratic form 
\begin{equation}\label{E:3m-1}
B(x_0,\dots, x_{m-1})=\sum_{i=0}^{m-1} (x_i-x_{i+1}+x_{i+2})^2
\end{equation}
is balanced for all $m\geq 4$ such that $m \not\equiv \pm 1 \pmod{6}$.
\end{lemma}

\begin{proof}
Clearly, $B$ is positive semi-definite and even positive definite if $m \not\equiv 0 \pmod 6$. 
Thus $B$ satisfies Condition $(0)$. We begin by verifying Condition $(r)$ for $2\leq r \leq m-2$. 
For any $(a_0, \dots, a_{m-1})\in \ZZ^m$ with $\sum_{i=0}^{m-1} a_i=r$, we
have 
\begin{multline*}
B(a_0,\dots,a_{m-1})\geq \sum_{i=0}^{m-1} \vert a_i-a_{i+1}+a_{i+2}\vert \\ \geq \vert \sum_{i=0}^{m-1} (a_i-a_{i+1}+a_{i+2})\vert 
=\vert \sum_{i=0}^{m-1} a_i \vert =r = B(v_r).
\end{multline*}
Hence Condition $(r)$ is satisfied for $2\leq r \leq m-2$. 

When $r=1$, we have $B(v_1)=3$. By the above, for any $(a_i)_{i=0}^{m-1}\in \ZZ^m$ with $\sum_{i=0}^{m-1} a_i=1$,
we have the estimate $B(a_0,\dots,a_{m-1})\geq 1$. Since $B(a_0,\dots,a_{m-1})$ is clearly odd, we conclude that
$B$ fails to satisfy Condition $(1)$ if and only if there 
is $(a_i)_{i=0}^{m-1}\in \ZZ^m$ such that $a_0-a_1+a_2=\pm 1$ and $a_i-a_{i+1}+a_{i+2}=0$ for all $i=1,\dots, m-1$. 
One checks that such a vector exists if and only if $m\equiv  1 \pmod{6}$ or $m\equiv -1 \pmod{6}$, in which case the vector is
either $(-1, -1, 0,1,1,1,0,\underbrace{-1,-1,0,1,1,0}_{\text{period}}, \cdots )$
or $(-1, -1, 0,1,0,\underbrace{-1,-1,0,1,1,0}_{\text{period}}, \cdots )$, up to a sign and cyclic shift. 
We are done by Remark \ref{R:finite-conditions}.
\end{proof}

We close this section with an open problem, whose solution will have an immediate 
bearing on the semiample cone of $\M_{0,n}$, as we shall see in the sequel. 
\begin{question} For an integer $m$, describe the convex cone of balanced cyclic quadratic forms in $m$ variables. 
Given an integral linear form $L(x_0,x_1,\dots,x_{m-1})$, determine whether the cyclic quadratic form
\[
Q(L)(x_0,\dots,x_{m-1}):=\sum_{k=0}^{m-1} L^2(x_k, x_{k+1}, \dots,x_{k+m-1})
\] 
is balanced. For example, for which $m$ is the following form balanced:
\[
\sum_{i=0}^{m-1} (x_i-x_{i+1}+x_{i+2}-x_{i+3}+x_{i+4})^2 \ ?
\]
\end{question}

\subsection{Quadratic forms of functions on cyclic groups}
\label{quadratic-from-function}
For a symmetric function $f\co \ZZ_m \ra \QQ$,
we define $q_f\co  \ZZ_m \ra \QQ$ by 
\begin{equation}\label{E:2nd-diff}
q_f(a):=\frac{1}{2}\bigl(f(a+1)+f(a-1)-2f(a)\bigr).
\end{equation}
\begin{remark}\label{R:F-basis}
Given a function $f$ as above, the number $q_f(a)$ can be interpreted as follows.
Consider the divisor $D:=\DD(\ZZ_m,f; (1)_{m})$ on $\M_{0,m}$. 
Then for the F-curve $F_{1,1,a-1}$ (we follow the notation of \cite[\S2.2.2]{agss}), we have
\[
D\cdot F_{1,1,a-1}=2q_f(a)+2f(1)-f(2).
\]
In other words, $q_f$ is determined (up to a constant) by the intersection 
numbers of $D$ with the collection of F-curves $\{F_{1,1,a-1}\}_{a=2}^{\lfloor m/2\rfloor}$,
which form a basis for $N_1(\M_{0,m}/\Sm_m)$ by \cite[Proposition 4.1]{agss}.
\end{remark}

We associate to $f$ a cyclic quadratic form defined by
\begin{equation}\label{E:quadratic}
Q_f(x_0,\dots, x_{m-1}):= \sum_{0\leq i,\, j\leq m-1} q_f(i-j) x_ix_j,
\end{equation}
and call $Q_f$ \emph{the associated quadratic form of $f$}. Note that 
$\sum_{k=0}^{m-1} q_f(k)=0$ and consequently by \eqref{E:shift}, we have
\begin{equation}\label{E:remark-Qf}
Q_f(x_0+c, \dots, x_{m-1}+c)=Q_f(x_0, \dots, x_{m-1}).
\end{equation}

We are interested in exploring when $Q_f$ is balanced or weakly balanced. 
We begin with a simple criterion for $Q_f$ to be weakly balanced.
\begin{lemma}\label{L:weakly-balanced}
$Q_f$ is weakly balanced if and only if for every $S \subset \{0,\dots, m-1\}$, we have 
\begin{equation}\label{E:weakly-balanced}
f(\vert S\vert)-f(0) \geq \frac{1}{2} \sum_{i,\, j\in S} \bigl( f(i-j-1)+f(i-j+1)-2f(i-j)\bigr).
\end{equation}
\end{lemma}
\begin{proof}
This follows immediately from Definition \ref{D:balanced} (b) by comparing the sum of entries in the principal 
minor of $Q_f$ given by $S$ to the sum of entries of the leading principal $\vert S\vert \times \vert S\vert$
minor of $Q_f$.
\end{proof}
\begin{remark}
It follows from Remark \ref{R:F-basis} that \eqref{E:weakly-balanced}
is equivalent to the condition that $\DD(f,\ZZ_m; (1)_{m})$ intersects certain explicit curve classes
on $\M_{0,m}$ non-negatively. It is not clear to us whether these curve classes have any geometric
significance, or whether they are effective. 
\end{remark}

\begin{corollary}\label{implications}
Suppose $Q_f$ is weakly balanced and $f(0)=0$. Then 
\[
f(a)+f(b)+f(c)+f(a+b+c)\geq f(a+b)+f(a+c)+f(b+c)
\]
for all $a,b,c\in \{0,\dots, m-1\}$ such that $a+b+c\leq m$.
\end{corollary}
\begin{proof}
Take $S=\{0,\dots, a-1\}\cup \{a+b,\dots, a+b+c-1\}$
in Lemma \ref{L:weakly-balanced}. 
\end{proof}
Unfortunately, we do not have a criterion for balancedness of $Q_f$. We note however that, by Remark \ref{R:PSD}, $Q_f$ is balanced if and only if 
\[
\widetilde{Q}_f=Q_f+2m\bigl(f(1)-f(0)\bigr)\bigl(x_0+\cdots+x_{m-1}\bigr)^2
\]
is positive semi-definite and balanced. 

\subsection{Cyclic effectivity of F-nef functions and balancedness of quadratic forms}


\begin{prop}\label{P:cyclic=balanced}
Consider a symmetric function $f\co \ZZ_m \to \QQ$ with $f(0)=0$.
\begin{enumerate}
\item The function $f$ is cyclically 
effective if and only if  the associated
quadratic form $Q_f$ is balanced. 
\item The function $f$ is weakly cyclically 
effective if and only if  the associated
quadratic form $Q_f$ is weakly balanced. 
\end{enumerate}
\end{prop}
Our proof of the above proposition relies on a simple construction of a weighting
on $\Gamma([n])$, called \emph{the cyclic weighting}, which we proceed to explain.

\subsubsection*{Cyclic weighting}

A cyclic weighting on $\Gamma([n])$ depends 
on a choice of a $G$-$n$-tuple $\vec{d}$ and a cyclic ordering of $[n]$, which we think of 
as a bijection $\sigma\co [n] \to P_n$, where $P_n$ is a regular $n$-gon. 
Before we construct the cyclic weighting in Lemma \ref{L:cyclic}, we need to introduce a bit 
of terminology and notation.


\begin{definition}\label{D:sigma-vector}
Given an $n$-tuple $(d_1,\dots,d_n)$ of residues modulo $m$ 
such that $\sum_{i=1}^n d_i \equiv 0 \pmod m$, a partition $I\sqcup J=[n]$, and 
a cyclic ordering $\sigma$ of $[n]$, 
we define the 
vector $\sigma_I(\vec{d})\in \ZZ^{m}$ as follows: 
\begin{enumerate}
\item Replace each $d_i$ by a positive integer representative of the residue.
Set $N=\sum_{i=1}^n d_i$.
\item Consider a regular $N$-gon $P_N$ and 
a correspondence
between $[n]$ and the vertices of $P_N$ such that 
\begin{enumerate}
\item Each $i$ corresponds to some $d_i$ contiguous vertices $S_i$ of $P_N$.
\item $\cup_{i=1}^n S_i = P_{N}$.
\item The order in which the $S_i$'s occur in $P_N$ as one goes in the clockwise
direction along $P_N$ is given by $\sigma$.
\end{enumerate}
\item For $j=0,\dots, m-1$, we define $z_j$ to be the number of vertices of $P_N$ 
that belong to $\cup_{i\in I} S_i$ 
and whose index is congruent to $j$ modulo $m$. 
At last, we set 
\[
\sigma_I(\vec{d}):=(z_0,\dots,z_{m-1}).
\]
\end{enumerate}
\end{definition}

\begin{remark} 
Of course,  $\sigma_I(\vec{d})=(z_0,\dots,z_{m-1})$ 
is only defined up to a multiple of $(1,1,\dots,1)$,
and up to a cyclic permutation. 
Note that $\sum_{i=0}^{m-1}z_i= \sum_{i\in I} d_i= \vec{d}(I) \pmod{m}$.
We also have $\sigma_I(\vec{d}) \equiv -\sigma_J(\vec{d}) \pmod{m}$
for a partition $I\sqcup J=[n]$. It is worth noting that if $I\sqcup J=[n]$ is $\sigma$-contiguous, 
then $\sigma_I(\vec{d})=\bigl((c+1)_{r}, (c)_{m-r}\bigr)$, where $\vec{d}(I)=mc+r$. 
\end{remark}

\begin{lemma}[Cyclic weighting]
\label{L:cyclic} Suppose $f\co \ZZ_m \to \QQ$ is a symmetric function 
and $Q_f$ 
is the associated quadratic form.
Then for any $\ZZ_m$-$n$-tuple $\vec{d}=(d_1,\dots, d_n)$
and any cyclic ordering $\sigma \co [n] \to P_n$,
there exists a unique weighting 
$w_\sigma$ on $\Gamma([n])$ such that 
for every $\sigma$-contiguous partition $I\sqcup J=[n]$ we have
\begin{equation}\label{E:cyclic-weight}
w_\sigma(I\mid J)=Q_f(\sigma_I(\vec{d}))=f(\vec{d}(I))-f(0).
\end{equation}
Moreover, for such $w_\sigma$, 
it continues to hold that
$w_\sigma(I\mid J)=Q_f(\sigma_I(\vec{d}))$ for every partition $I\sqcup J=[n]$.
\end{lemma}
\begin{proof}
It is clear that specifying the flow across all $\sigma$-contiguous partitions determines $w_\sigma$
uniquely. For example, if $i$ and $j$ are adjacent in $P_n$, then we must have
\[
w_{\sigma}(i \sim j)=\Bigl(w_\sigma(i)+w_\sigma(j)-w_\sigma(\{i,j\} \mid [n]- \{i,j\})\Bigr)/2,
\]
and the weights of all the remaining edges are determined similarly. 

It remains to prove the existence of the requisite $w_\sigma$.
To begin, as in Definition \ref{D:sigma-vector} (1), replace each $d_i$ by a positive integer representative of 
the residue and set $N=\sum_{i=1}^n d_i$. We have $m \mid N$. 
Consider the regular $N$-gon $P_N$ and a correspondence
between $[n]$ and the vertices of $P_N$ as in Definition \ref{D:sigma-vector} (2):
The vertex $i\in [n]$ corresponds to $d_i$ adjacent vertices $S_i\subset P_{N}$ 
in such a way that $\cup_{i=1}^n S_i = P_{N}$ 
and $S_i$'s occur in the clockwise order in $P_{N}$ given by $\sigma$. 
Let $q_f\co \ZZ_m \to \QQ$ be as in \eqref{E:2nd-diff}.
We define the weight function $\widetilde{w}$ on $P_N$ by 
\[
\widetilde{w}(i\sim j):=-q_f(i-j)=\frac{1}{2}\bigl(2f(i-j)-f(i-j-1)-f(i-j+1)\bigr). 
\]
Since $f$ is a function on $\ZZ_m$, $\widetilde{w}$ is invariant under rotations of $P_N$ 
by the angle $2\pi/m$. In particular, 
\begin{itemize}
\item The weight $\widetilde{w}(i\sim j)$ depends only on $i-j \pmod{m}$.
\item All vertices have the same $\widetilde{w}$-flow.
\item The $\widetilde{w}$-flow across a contiguous partition $A\sqcup B$ of $P_N$ depends only on $\vert A\vert$. 
 \end{itemize}
We now compute the $\widetilde{w}$-flow through the vertex $1\in P_N$:
\begin{multline*}
\widetilde{w}(1)=\sum_{j=2}^{n} \widetilde{w}(1\sim j)=-\sum_{j=2}^n q_f(j-1)  =\frac{1}{2}\sum_{j=2}^n \bigl(2f(j-1)-f(j)-f(j-2)\bigr) \\ =f(1)-f(0)=Q_f(1,0,\dots,0). 
\end{multline*}

Suppose that $A \sqcup B$ is an arbitrary partition of $P_{N}$. 
For $k=0,\dots, m-1$, let $z_j$ be the number of vertices in $A$ whose index is congruent to $j \pmod{m}$.
From the definition of $\widetilde{w}$, we see that 
\begin{multline}\label{E:flow-across}
\widetilde{w}(A\mid B)=-\sum_{0\leq i, \, j\leq m-1} q_f(i-j) z_i (m-z_j) \\ =-m\left(\sum_{i=0}^{m-1}z_i \right)\left(\sum_{k=0}^{m-1} q_f(k)\right)
+\sum_{0\leq i, \, j\leq m-1} q_f(i-j) z_i z_j  =Q_f(z_0,\dots, z_{m-1}).
\end{multline}

It remains to observe that the weighting $\widetilde{w}$ on $P_{N}$
induces a weighting on $\Gamma([n])$
that satisfies the conditions of the lemma. 
Namely, 
 we define the $w_{\sigma}$-weight of $(i\sim j)$ in $E([n])$ to be the 
sum of $\widetilde{w}$-weights of the edges in $E(P_{N})$ joining $S_i$ and $S_j$:
\[
w_{\sigma}(i\sim j):= \sum_{a\in S_i, b\in S_j} \widetilde{w}(a\sim b).
\]
It follows from \eqref{E:flow-across} and Definition \ref{D:sigma-vector} that 
for every partition $I\sqcup J=[n]$,
we have
\[
 w_\sigma(I\mid J)=Q_f(\sigma_I(\vec{d})).
 \]
Since $Q_f(\sigma_I(\vec{d}))=f(\vec{d}(I))-f(0)$ if $I\sqcup J$ is $\sigma$-contiguous, we are done. 
\end{proof}
\begin{remark}
That flows across all $\sigma$-contiguous partitions determines $w_\sigma$
uniquely is equivalent to the fact that the collection $\mathcal{A}_\sigma$ of divisors defined in 
Remark \ref{R:nonadjacent} forms a basis of $\Pic(\M_{0,n})$. This follows immediately from 
the Effective Boundary Lemma \ref{L:main}. In view of this, the main contribution of Lemma \ref{L:cyclic} 
is the computation of $w_\sigma$-flows across \emph{all} partitions of $[n]$ in terms of the values 
of the quadratic form $Q_f$. 
\end{remark}

We are now ready to establish the equivalence between 
cyclic effectivity of $f\co \ZZ_m \to \QQ$ and the balancedness of $Q_f$.
\begin{proof}[Proof of Proposition \ref{P:cyclic=balanced}]
Suppose $f\co \ZZ_m \to \QQ$ is a symmetric function with $f(0)=0$.
Let $\vec{d}=(d_1,\dots,d_n)$ be a $\ZZ_m$-$n$-tuple,
where we assume each $d_i$ to be a positive integer.
By definition, $f$ is cyclically effective with respect to 
$\vec{d}$ if and only if for every cyclic ordering $\sigma\co [n]\to P_n$, the weighting 
$w_\sigma$ constructed in Lemma \ref{L:cyclic}
satisfies 
\begin{equation}\label{E:cyclic-balanced}
w_\sigma(I\mid J) \geq f(\vec{d}(I))
\end{equation} for every partition $I\sqcup J=[n]$. (Note that the equality 
is already satisfied for all $\sigma$-contiguous partitions by the construction of $w_\sigma$.)
But by Lemma \ref{L:cyclic}, we have 
\[
w_\sigma(I\mid J)=Q_f(\sigma_{I}(\vec{d})),
\]
where $\sigma_{I}(\vec{d})=(z_0,\dots, z_{m-1})\in \ZZ^{m}$ is a vector such that
\[
z_0+\cdots+z_{m-1}=\vec{d}(I).
\]
Suppose $\vec{d}(I)=mc+r$. Then 
\[
Q_f\bigl((1+c)_{r}, (c)_{m-r}\bigr)=f(r)=f(\vec{d}(I)).
\]
We conclude that \eqref{E:cyclic-balanced} holds for all $\vec{d}$ if and only if
\[
Q_f(z_0,\dots,z_{m-1}) \geq Q_f\bigl((1+c)_{r}, (c)_{m-r}\bigr)
\]
for all $(z_0,\dots,z_{m-1})\in \ZZ_{>0}^{m}$ satisfying $\sum_{i=0}^{m-1} z_i=\vec{d}(I)$. 
However, this is precisely the condition for $Q_f$ to be balanced.\footnote{We have used Remark \ref{R:finite-conditions} to restrict to positive $z_i$'s here.} 
This finishes the proof of the first part. 

The second part follows by the same argument after
minor modifications. Namely, we observe that the $\ZZ_m$-$n$-tuples $\vec{d}$
that need to be considered when verifying weak cyclic effectivity give rise
to the vectors $\sigma_I(\vec{d})$ whose every coordinate is either $0$ or $1$.  Hence
weak balancedness of $Q_f$ is sufficient to conclude.
\end{proof}

\begin{theorem}\label{T:balanced}
Suppose $f\co \ZZ_m \ra \ZZ$ is a symmetric function with $f(0)=0$ such that the associated 
quadratic form $Q_f$ is balanced. Then 
\[
\DD\bigl(\ZZ_m, f; (d_1,\dots, d_n)\bigr)
\]
is boundary semiample on $\M_{0,n}$ for all $\ZZ_m$-$n$-tuples $(d_1,\dots, d_n)$. 
\end{theorem}
\begin{proof}
We have that $f$ is cyclically effective by Proposition 
\ref{P:cyclic=balanced}(1). The result follows by Theorem \ref{P:semiample}.
\end{proof}

\begin{example}[Type A level 1 conformal blocks divisors]
\label{E:typeA}
The quadratic form associated to the standard function $A_m$ (see Example \ref{E:standard}) is 
\[
Q_{A_m}(x_0,\dots,x_{m-1})=m(x_0^2+x_1^2+\cdots+x_{m-1}^2)-\left(\sum_{i=0}^{m-1} x_i\right)^2.
\]
Note that $Q_{A_m}$ is balanced because $x_0^2+x_1^2+\cdots+x_{m-1}^2$ is balanced. 
Every $\sl_m$ level $1$ conformal blocks divisor (see \cite{fakh} definition, and \cite{gian,gian-gib} for 
GIT interpretation and other wonderful properties of these divisors)
is a multiple of $\DD\bigl(\ZZ_m, A_m; (d_1,\dots, d_n)\bigr)$ for some $d_1,\dots, d_n \in \ZZ_m$ 
such that $m\mid \sum_{i=1}^{n} d_i$ by \cite[Theorem A and Proposition 4.8]{fedorchuk-cyclic}. 
Thus Theorem \ref{T:balanced} gives a new proof of semiampleness for these divisors, 
valid over $\spec \ZZ$.
\end{example}

\subsection{The cyclic semiample cone of $\M_{0,n}$}
The condition for a quadratic form to be balanced is quite stringent and checking whether a given quadratic form is balanced can be quite difficult. On the other hand, weak balancedness of a quadratic form
can be easily checked. Thus we single out those symmetric divisors on $\M_{0,n}$
whose associated quadratic form is weakly balanced, in the process obtaining  
a simple sufficient criterion for a symmetric divisor on $\M_{0,n}$ to be semiample. 

Suppose $D=\sum_{i=2}^{\lfloor n/2\rfloor} a_i\Delta_i$ is a symmetric divisor on $\M_{0,n}$.
Set $a_0=a_1=0$ and $a_{i}=a_{n-i}$ for $\lfloor n/2\rfloor+1\leq i \leq n-1$.
Define a function $q_D\co \ZZ_n \ra \QQ$ by 
\[
q_D(i):=\frac{1}{2}\bigl(2a_i-a_{i-1}-a_{i+1}\bigr).
\]
and a quadratic form
\begin{equation}\label{E:quadratic}
Q_D(x_0,\dots, x_{n-1}):= \sum_{0\leq i,\, j\leq n-1} q_D(i-j) x_ix_j.
\end{equation}
We call $Q_D$ \emph{the associated quadratic form of $D$}. 

\begin{theorem}[Cyclic Semiampleness Criterion]\label{T:semiample} 
Let $D=\sum_{i=2}^{\lfloor n/2\rfloor} a_i\Delta_i$ be a symmetric divisor on $\M_{0,n}$.
Suppose that for every $S\subset \{0,\dots, n-1\}$,
the following inequality holds: 
\begin{equation}\label{E:weakly}
2a_{\vert S\vert} \geq  \sum_{i,\, j\in S} \bigl(2a_{i-j}-a_{i-j+1}-a_{i-j-1}\bigr),
\end{equation} 
where as before $a_0=a_1=0$ and $a_i=a_{n-i}$.
Then $D$ is semiample. Moreover, there exists a base point free linear subsystem of $\vert 2D\vert$ 
of dimension at most $(2n-5)!!$ 
\end{theorem}

\begin{proof}
By Lemma \ref{L:weakly-balanced} applied to the symmetric function $p_D\co \ZZ_n \ra \ZZ$ 
defined in \eqref{E:pd-function},
the associated quadratic form $Q_D=Q_{p_D}$ is weakly balanced.  
Proposition \ref{P:cyclic=balanced}(2) now shows that $f:=2p_{D}$  is a 
weakly cyclically effective function, with values in even integers. 
Since
\[
\DD\bigl(\ZZ_n, f; (1,\dots, 1)\bigr)=2D,
\]
the semiampleness follows at once from Theorem \ref{P:weakly-semiample}.
The claim about base point free linear system follows once we observe that
because the cyclic weighting $w_\sigma$ constructed in Lemma \ref{L:cyclic} is
integer-valued when the function $f$ takes values in even integers,
we actually obtain $\ZZ$-linear equivalence in Lemma \ref{L:T-effective}(2).
\end{proof}

\begin{definition}\label{D:cyclic-semiample}
We say that the symmetric divisor $D$ on $\M_{0,n}$ is \emph{cyclic semiample} if 
$D$ satisfies the criterion of Theorem \ref{T:semiample}. Clearly, 
cyclic semiample divisors form a finite polyhedral subcone in $\NS(\M_{0,n})^{\Sm_n}$,
which we call \emph{the cyclic semiample cone}. 
\end{definition}

\begin{prop}
The cyclic semiample cone is a full-dimensional subcone in $\NS(\M_{0,n})^{\Sm_n}$
\end{prop}
\begin{proof}
Consider the following expression for a well-known ample divisor $\psi-\Delta$ on $\M_{0,n}$:
\[
\psi-\Delta=\sum_{i=2}^{\lfloor n/2\rfloor} \frac{i(n-i)-(n-1)}{n-1} \Delta_i.
\]
The coefficients of $\Delta_i$'s above \emph{strictly} satisfy Inequalities \eqref{E:weakly}. 
It follows that Theorem \ref{T:semiample} gives a
full-dimensional subcone in $\NS(\M_{0,n})^{\Sm_n}$ consisting of semiample divisors.
\end{proof}

\begin{remark}
\label{R:level-one}
 By Example \ref{E:typeA}, all symmetric $\sl_n$ level $1$ conformal blocks divisors 
lie in the cyclic semiample cone of $\M_{0,n}$. These $\nhalf$ divisors are known to be extremal rays
of the nef cone by \cite[Theorem 1.2]{agss}. It follows from \cite[Corollary 1.3]{kazanova}
that all $\Sm_n$-invariant type A conformal blocks vector bundles of rank $1$ on $\M_{0,n}$ lie in the 
cyclic semiample cone. 
Computer computations show that the 
cyclic semiample cone shares many more than $\nhalf$ extremal rays with the symmetric F-cone of $\M_{0,n}$. 
For example, when $n=20$, there are $739$ extremal rays of the symmetric F-cone and 
$60$ of them are semiample by Theorem \ref{T:semiample}. 
\end{remark}

\subsection{Geometric interpretation of the cyclic semiample cone}
\label{S:geometric} Let $R$ be $\ZZ$ or $\QQ$.
Suppose $X$ is a smooth projective variety with a finitely generated Picard group. 
We say that a collection of effective divisors $\B=\{D_1,\dots, D_\rho\}$ on $X$ is 
an \emph{$R$-basis} if the classes of $D_1,\dots, D_\rho$ 
form a basis of $\Pic(X)\otimes R$. 

Suppose $\B=\{D_1,\dots, D_\rho\}$ is an $R$-basis. Define \emph{the complement of $\B$}
to be the following open subset of $X$:
\[
U(\B):=X\setminus \bigcup_{i=1}^{\rho} D_i.
\]

Every $R$-basis $\B$ defines $\rho$ $R$-linear functions $c^\B_i\co \Pic(X)\otimes R \to R$, 
where $c^\B_i(L)$ for $L\in \Pic(X)\otimes R$ are uniquely defined by 
\[
L=\cO\left(\sum_{i=1}^\rho c^\B_iD_i\right).
\]
Let $E(\B)\subset \Pic(X)\otimes R$ be the convex cone of those line bundles whose coefficients
in the basis $\B$ are all non-negative. Namely, 
\[
E(\B):=\left\{L\in \Pic(X)\otimes R \mid c^\B_i(L)\geq 0\right\}.
\]

\begin{definition}\label{D:easy} We say that a collection of $R$-bases $\B_1, \dots, \B_N$ is 
\emph{exhaustive} if
\begin{equation}\label{E:empty-intersection}
U(\B_1)\cup U(\B_2)\cup \cdots \cup U(\B_N)=X.
\end{equation}
Given an exhaustive collection $\B_1, \dots, \B_N$, we define 
\[
\Nef(\B_1, \dots, \B_N):=E(\B_1)\cap E(\B_2)\cap \cdots \cap E(\B_N).
\]
\end{definition}

\begin{lemma}\label{L:exhaustive-nef}
 Suppose $\B_1, \dots, \B_N$ is an exhaustive collection.
Then $\Nef(\B_1,\dots,\B_N)\subset \Nef(X)$. Moreover, if $\Pic(X)$ is
in addition torsion-free, then $\Nef(\B_1,\dots,\B_N)\subset \Semiample(X)$.
\end{lemma}
\begin{proof} This is obvious. By definition, $L\in E(\B_i)$ if and only if $L$ can be represented
modulo torsion by an effective divisor with support in $\B_i$. Condition \eqref{E:empty-intersection}
says that the supports of $\B_i's$ do not intersect.
\end{proof}

The following observation is known to expert (we learned about it from Jenia Tevelev): 
\begin{lemma}\label{L:non-adjacent-exhaustive} 
The collection of non-adjacent bases on $\M_{0,n}$ (see Remark \ref{R:nonadjacent}) taken over all cyclic orderings of $[n]$ is exhaustive.
\end{lemma}

\begin{proof}
This is just the reformulation of the implication 
\[
\text{cyclically effective} \Rightarrow \text{tree-effective}
\]
from Lemma \ref{L:implication-1}. Namely, if $T$ is an $[n]$-labeled tree, consider its embedding into the 
plane so that the leaves are identified with the vertices of $P_n$.
The resulting cyclic ordering $\sigma\co [n] \to P_n$ 
satisfies the property that every $T$-partition of $[n]$ is $\sigma$-contiguous. 
In particular, the generic point of $M_T$ is not in the support of the divisors in $\A_{\sigma}$. 
\end{proof}

\begin{lemma} The cyclic semiample cone of $\M_{0,n}$ is 
$\Easy(\{\A_\sigma \mid \text{$\sigma$ is a cyclic ordering of $[n]$}\})$.
\end{lemma}
\begin{proof}
This follows by unwinding Definitions  \ref{D:cyclic-semiample} and \ref{D:easy}.
\end{proof}

\section{Tree-effectivity and another semiampleness criterion}
\label{S:2nd-semiampleness}
In the previous section, we have explored the strongest of the three effectivity conditions
on a symmetric function, namely cyclic effectivity. In particular, we obtained an ``if and only 
if'' characterization of cyclic effectivity of a function $f$ in terms of the balancedness
of the associated quadratic form $Q_f$. In this section, we turn our attention to the tree-effectivity,
but with more modest results. Namely, we prove a sufficient criterion for a function to be 
tree-effective (resp., weakly tree-effective). This leads to another sufficient
criterion for semiampleness of symmetric divisors on $\M_{0,n}$. On one hand,
this criterion allows to exhibit boundary semiample divisors on $\M_{0,n}$ that 
are not cyclic semiample; on the other hand,
this criterion is too stringent because many cyclically effective functions do not satisfy it,
whereas every cyclically effective function is tree-effective by Remark \ref{R:implications}.
Both of these phenomena are illustrated by Table \ref{table-experiments} in Section \ref{S:examples}.
\begin{prop}\label{P:tree-effective} 
\hfill
\begin{enumerate}
\item Suppose $f\co G \to \QQ_{\geq 0}$ is a subadditive symmetric function satisfying 
\begin{equation}\label{E:sufficient-strong}
\begin{aligned}
&{}\quad (f(b)-f(a)+f(a+b))f(A)\\ &+(f(a)-f(b)+f(a+b))f(B)\\ &+(f(a)+f(b)-f(a+b))f(a+b)\\ &\geq 2f(B+b)f(a+b).
\end{aligned}
\end{equation}
for all $a,b,A,B\in G$ such that $a+b+A+B=0\in G$.
Then $f$ is tree-effective. 

\item Suppose $f\co \ZZ_n \to \QQ_{\geq 0}$ is a subadditive symmetric function satisfying 
\eqref{E:sufficient-strong}
for all non-negative integers $a,b,A,B$ such that $a+b+A+B=n$.
Then $f$ is weakly tree-effective. 
\end{enumerate}
\end{prop}
\begin{proof}
We begin by establishing the first part.

Let $\vec{d}=(d_1,\dots,d_n)$ be a $G$-$n$-tuple. 
We need to verify that $f$ is $T$-effective with respect to $\vec{d}$
for every $[n]$-labeled unrooted binary tree $T$. We proceed by induction on $n$ starting with the base
case $n=3$, which holds by Remark \ref{R:n=34}. 

Suppose $n\geq 4$. Take an internal node $p$ of $T$ from which two leaves emanate.
Without loss
of generality, these leaves are $n-1$ and $n$. 
Let $v=\overline{pq}$ be the internal edge of $T$ containing $p$. 
We define $T_0$ to be the $[n-1]$-labeled tree obtained 
by contracting the leaves $n-1$ and $n$ of $T$ into $p$, so that $p$ becomes the $(n-1)^{st}$ leaf of $T_0$. By our inductive assumption, $f$ is $T_0$-effective
with respect to $\vec{d}_0:=(d_1,\dots, d_{n-2}, d_{n-1}+d_{n})$. 
Equivalently, there exists a weighting $w_0$ on $\Gamma([n-1])$ such that 
\begin{enumerate}
\item $w_0(i)=f(d_i)$ for all $i=1,\dots, n-2$. 
\item $w_0(n-1)=f(d_{n-1}+d_{n})$.
\item $w_0(I\mid J)\geq f(\vec{d_0}(I))$ for all partitions $I\sqcup J=[n-1]$. 
\item $w_0(I\mid J)= f(\vec{d_0}(I))$ for all $T_0$-partitions $I\sqcup J=[n-1]$.
\end{enumerate}
Set $a:=d_{n-1}$ and $b:=d_n$,
We define a weighting $w$ on $\Gamma([n])$ as follows:
\begin{equation}\label{E:new-weight}
\begin{aligned}
w(i\sim j)&=w_0(i\sim j) \text{ if $i,j\in [n-2]$}, \\
w((n-1)\sim n)&=\frac{f(a)+f(b)-f(a+b)}{2},\\
w(n\sim i)&=\frac{f(a)-f(b)+f(a+b)}{2f(a+b)}w_0((n-1)\sim i), \text{ for $i\in [n-2]$},\\
w((n-1)\sim i)&=\frac{f(b)-f(a)+f(a+b)}{2f(a+b)}w_0((n-1)\sim i), \text{ for $i\in [n-2]$}. \\
\end{aligned}
\end{equation}

Clearly, the $w$-flow through the vertex $n$ of $T$ is 
\begin{align*}
w(n)&=\frac{f(a)-f(b)+f(a+b)}{2f(a+b)}w_{0}(n-1)+\frac{f(a)+f(b)-f(a+b)}{2} \\
&=\frac{f(a)-f(b)+f(a+b)}{2f(a+b)}f(a+b)+\frac{f(a)+f(b)-f(a+b)}{2}=f(a)=f(d_n),
\end{align*}
and the $w$-flow through the vertex $n-1$ of $T$ is 
\[
w(n-1)=\frac{f(b)-f(a)+f(a+b)}{2f(a+b)}w_{0}(n-1)+\frac{f(a)+f(b)-f(a+b)}{2}=f(d_{n-1}).
\]
We also have \[
w(i)=w_0(i)=f(d_i), \quad \text{ for all $i\in [n-2]$}.
\]
Thus $w$ satisfies Condition (a) of Definition \ref{D:effective-with-respect} (2).

Suppose $I\sqcup J=[n]$ is a partition such that $n-1$ and $n$ lie in the same subset,
say, $J$. Then from \eqref{E:new-weight}, we obtain $w(I\mid J)=w_0(I\mid [n-1]- I)$.
It follows from $T_0$-effectivity of $f$ that
\[
w(I\mid J)=w_0(I\mid [n-1]- I)\geq f(\vec{d}_0(I))=f(\vec{d}(I)),
\]
with the equality achieved when $I\sqcup J=[n]$ is a $T$-partition. Thus $w$ satisfies Condition (c) of Definition \ref{D:effective-with-respect} (2).

Consider now
a partition of $[n]$ given by $I=n\cup I'$ and $J=(n-1)\cup J'$,
where $I'\sqcup J'=[n-2]$. We set
\begin{align*}
A&:=\vec{d}(I')=\sum_{i\in I'} d_i, \\
B&:=\vec{d}(J')=\sum_{j\in J'} d_j.
\end{align*}
Then $a+b+A+B=0 \in G$. We proceed to estimate the $w$-flow across $I\sqcup J$,
which we compute using \eqref{E:new-weight}:

\begin{align*}
w(I\mid J)&=w((n-1)\sim n)+\sum_{i\in I', j \in J'} w(i\sim j)+\sum_{i\in I'} w(i\sim (n-1))+\sum_{j\in J'} w(j\sim n) \\
&=w((n-1)\sim n)+\sum_{i\in I', j \in J'} w_0(i\sim j) \\ &+\frac{f(b)-f(a)+f(a+b)}{2f(a+b)} \sum_{i\in I'} w_0(i\sim (n-1))+\frac{f(a)-f(b)+f(a+b)}{2f(a+b)}\sum_{j\in J'} w_0(j\sim n)\\
&=\frac{f(a)+f(b)-f(a+b)}{2}\\ &+\frac{f(b)-f(a)+f(a+b)}{2f(a+b)}w_0(I' \mid J'\cup \{n-1\})+\frac{f(a)-f(b)+f(a+b)}{2f(a+b)} w_0(J' \mid I'\cup \{n-1\}) \\
&\geq \frac{f(a)+f(b)-f(a+b)}{2}+\frac{f(b)-f(a)+f(a+b)}{2f(a+b)}
f(\vec{d_0}(I'))+\frac{f(a)-f(b)+f(a+b)}{2f(a+b)}f(\vec{d_0}(J'))\\
&=\frac{f(a)+f(b)-f(a+b)}{2}+\frac{f(b)-f(a)+f(a+b)}{2f(a+b)}f(A)+\frac{f(a)-f(b)+f(a+b)}{2f(a+b)}f(B) \\
&\geq f(B+b)=f(\vec{d}(I)).
\end{align*}
Thus $w$ satisfies Condition (c) of Definition \ref{D:effective-with-respect} (2).
This finishes the proof of the first part. The second part follows by exactly the same argument,
noting that the induction step replaces a $k$-tuple $\vec{d}=(d_1,\dots, d_k)$ of non-negative integers 
satisfying $\sum_{i=1}^k d_i=n$ by a $(k-1)$-tuple $\vec{d}_0=(d_1,\dots, d_{k-2}, d_{k-1}+d_{k})$
 of non-negative integers 
satisfying the same condition. 
\end{proof}

\begin{theorem}\label{T:2nd-semiample}
Let $D=\sum_{r=2}^{\lfloor n/2\rfloor} a_r\Delta_r$ be a symmetric divisor on $\M_{0,n}$
and $p_D\co \ZZ_n \to \ZZ$ be the function defined by \eqref{E:pd-function}.
Suppose that for some 
\[
\lambda \geq \lambda_{\fnef}(D) 
\]
we have that the function $f_{\lambda}\co \ZZ_n \to \QQ$ defined by
$f_\lambda=\lambda A_n+p_D$ (i.e., $f_{\lambda}(r)=\lambda r(n-r)-a_r$)
satisfies \eqref{E:sufficient-strong}
for all non-negative integers $a,b,A,B$ such that $a+b+A+B=n$. Then $D$ is boundary semiample.
\end{theorem}
\begin{proof}
By the assumption on $\lambda$, $f_{\lambda}$ is F-nef (cf. the proof of Lemma \ref{L:every-F-nef}). Hence
$f_{\lambda}$ takes only non-negative values and $f_\lambda$ is subadditive by Lemma \ref{L:subadditivity}. 
By Proposition \ref{P:tree-effective}(2), we have that $f_{\lambda}$ is weakly tree-effective.
Noting that
\[
D=\DD\bigl(\ZZ_n, f_{\lambda}; (\underbrace{1,\dots,1}_{n})\bigr),
\]
we conclude that $D$ is boundary semiample by Theorem \ref{P:semiample}.
\end{proof}

\section{New nef divisors on $\M_{0,n}$}
\label{S:new-nef}


In Proposition \ref{P:cyclic=balanced}, we showed that
a balanced cyclic quadratic form $Q_f(x_0,\dots,x_{m-1})$ corresponds to a cyclically effective F-nef function
$f\co \ZZ_m \to \QQ$ with $f(0)=0$, and so gives rise to an infinitude of 
boundary semiample divisors (Theorem \ref{T:balanced}).
In this section, we prove that if $Q_f$ (or, equivalently, $f$) satisfies some further positivity conditions,
then we can construct a new F-nef function $\widetilde{f}$ out of $f$ in such a way that $\widetilde{f}$ is effective.
This procedure gives rise to new infinite families of nef divisors $\DD(\ZZ_m, \widetilde{f}; \vec{d})$.

\begin{definition}\label{D:tilde-f} 
Suppose $f\co G \ra \QQ$ is a symmetric function with $f(0)=0$. Set 
\begin{align*}
m(f):=\frac{1}{2}\min \{ 2f(a)+2f(b)-f(a+b)-f(a-b) \mid a,b\in G \}. \\
\end{align*}
We define 
$\widetilde{f}\co G \ra \QQ$ by 
\begin{equation}
\widetilde{f}(k)=\begin{cases} f(k) & \text{for $k\neq 0$}, \\
m(f) & \text{for $k=0$}. 
\end{cases}
\end{equation}
\end{definition}

The point of this definition is that for an F-nef function $f\co G \to \QQ$ with $f(0)=0$, 
we can increase the value of the function at $0$  
up to $m(f)$ without losing the F-nefness property:
\begin{lemma}
If $f\co G \to \QQ$, with $f(0)=0$, is F-nef, then $\widetilde{f}\co G \to \QQ$ is also F-nef.
\end{lemma}
\begin{proof}
This follows easily from Definition \ref{D:F-nef}.
\end{proof}

\begin{prop}\label{P:new-nef} Let $m\geq 3$. Suppose $f\co \ZZ_m \ra \QQ$ is a cyclically effective F-nef function
such that the balanced associated quadratic form $Q_f$ satisfies the following additional condition:
\begin{enumerate}
\item[($\dagger$)] The minimum value of $Q_f$ on the set 
\[
\{(x_0,\dots,x_{m-1}) \in \ZZ^m \setminus (0,\dots, 0) \mid \sum_{i=0}^{m-1}x_i = 0\}
\]
is at least $m(f)$ from Definition \ref{D:tilde-f}.
\end{enumerate} 
Then $\widetilde{f}\co \ZZ_m \ra \QQ$ is effective. Moreover, for any $\ZZ_m$-$n$-tuple $\vec{d}$,
we have that
$\widetilde{f}$ is $T$-effective with respect to $\vec{d}$ for every $[n]$-labeled tree whose
all $T$-partitions $I\sqcup J=[n]$ satisfy $\vec{d}(I)\neq 0\in \ZZ_m$. 
\end{prop}

Before we give a proof of the above proposition, we formulate its implications for
the nef cone of $\M_{0,n}$:

\begin{theorem}
\label{T:new-nef}
 Let $f\co \ZZ_m \ra \QQ$ be as in Proposition \ref{P:new-nef}. Then
for any $\ZZ_m$-$n$-tuple $\vec{d}=(d_1,\dots,d_n)$, the 
divisor 
 \begin{equation}
\DD\bigl( \ZZ_m, \widetilde{f} ; \vec{d}\bigr)
=\sum_{i=1}^n \widetilde{f}(d_i)\psi_i-\sum_{I,J}\widetilde{f}(\vec{d}(I)) \Delta_{I,J}
=\DD\bigl(\ZZ_m, f; \vec{d}\bigr)- \widetilde{f}(0)\sum_{\substack{I\sqcup J=[n] \\ m\mid \vec{d}(I)}} \Delta_{I,J}
\end{equation}
is stratally effective boundary, hence nef, on $\M_{0,n}$.
Moreover, $\DD\bigl( \ZZ_m, \widetilde{f}; \vec{d}\bigr)$
is base point free away from $\sum_{\substack{I\sqcup J=[n] \\ m\mid \vec{d}(I)}} \Delta_{I,J}$.
\end{theorem}

\begin{proof} The first part follows from effectivity of $\widetilde{f}$ (by Proposition \ref{P:new-nef}) and 
Theorem \ref{P:semiample}. For the second part,
denote 
\[
\Delta_{0}:=\sum_{\substack{I\sqcup J=[n] \\ m\mid \vec{d}(I)}}  \Delta_{I,J}.
\]
Let $T$ be the dual graph of a $0$-dimensional boundary $B$ on $\M_{0,n}$ that
does not lie in $\Delta_{0}$. Then by Proposition \ref{P:new-nef}, there is an effective linear
combination of the boundary that avoids $B$ and is linearly equivalent to 
$\DD\bigl(\ZZ_m, \widetilde{f}; (d_1,\dots, d_n)\bigr)$. The claim now follows from the observation that any  
boundary stratum not entirely lying in $\Delta_{0}$ contains in its closure
some $0$-dimensional boundary stratum $B$ that also
does not lie in $\Delta_{0}$ (this is where we use $m\geq 3$). 
\end{proof}

\begin{proof}[Proof of Proposition \ref{P:new-nef}]
Let $\vec{d}=(d_1,\dots,d_n)$ be a $\ZZ_m$-$n$-tuple.
To prove effectivity of $\widetilde{f}$, we first reduce to the case where no $d_i$ is $0$ in $\ZZ_m$. Indeed, 
if $d_n=0$, then effectivity of $\widetilde{f}$ with respect to $(d_1,\dots,d_{n-1})$ implies
the effectivity of $\widetilde{f}$ with respect to $\vec{d}$ by considering any triangle in $\Gamma([n])$
with a vertex at $n$, and with edges emanating from $\widetilde{f}$ having weight $\widetilde{f}(0)/2$
and the remaining edge having weight $-\widetilde{f}(0)/2$. 

Second, we observe that if $\vec{d}$ contains no $0$'s, then there exists an $[n]$-labeled unrooted
binary tree whose
all $T$-partitions $I\sqcup J=[n]$ satisfy $\vec{d}(I)\neq 0\in \ZZ_m$. 
Thus to prove effectivity of $\widetilde{f}$ with respect to $\vec{d}$, it remains to prove
$T$-effectivity of $\widetilde{f}$ with respect to $\vec{d}$ for all such trees. 
This is achieved by constructing a $T$-cyclic weighting, which we now describe. 

To begin, we recall that 
any planar embedding of an $[n]$-labeled tree $T$ defines a cyclic ordering $\sigma\co [n]\to P_n$,
obtained by going around the embedded graph in the clockwise direction. 
Because every $T$-partition of $[n]$ gives a $\sigma$-contiguous partition of $[n]$, it follows
from cyclic effectivity of $f$ that
the cyclic weighting $w_{\sigma}$ on $\Gamma([n])$ constructed in Lemma \ref{L:cyclic}
satisfies 
\[
w_\sigma(I\mid J)\geq f(\vec{d}(I)),
\]
with the equality holding for every $T$-partition $I\sqcup J =[n]$.
There are $2^{n-3}$ distinct planar embedding of $T$. 
We define the \emph{$T$-cyclic weighting $w_T$} on $\Gamma([n])$
to be the average of all cyclic weightings $w_{\sigma}$ on $\Gamma([n])$ taken 
over all planar embeddings of $T$, as described above. We proceed to describe
the properties of $w_T$:

\begin{claim*}[Property 1] We have that $w_T(I\mid J) \geq f(\vec{d}(I))$ for every 
partition $I\sqcup J$. Thus $w_T(I\mid J) \geq \widetilde{f}(\vec{d}(I))$
for every partition $I\sqcup J=[n]$ satisfying $\vec{d}(I)\neq 0 \in \ZZ_m$.  
\end{claim*}

\begin{claim*}[Property 2] If $e$ is an edge of $T$ and $I\sqcup J=[n]$
is the $e$-partition satisfying $\vec{d}(I)= 0 \in \ZZ_m$, 
then $w_T(I\mid J)=Q_f(0,\dots,0)=0$.
\end{claim*}

\begin{claim*}[Property 3]
Suppose all $T$-partitions $I\sqcup J=[n]$ satisfy $\vec{d}(I)\neq 0 \in \ZZ_m$. 
We have that $w_T(I\mid J)\geq \widetilde{f}(0)$
for any partition $I\sqcup J=[n]$
such that $\vec{d}(I)= 0 \in \ZZ_m$.
\end{claim*}

\begin{proof} 
The first two claims follow from the properties of the cyclic weighting $w_{\sigma}$
given in Lemma \ref{L:cyclic} and the cyclic effectivity of $f$. We proceed to prove the third claim.

Fix $I\sqcup J=[n]$ with $\vec{d}(I)=0\in \ZZ_m$. By assumption on $T$, $I\sqcup J$ is not a $T$-partition. 
Let $e_0=\overline{pq}$ be the internal edge of $T$ satisfying the following property.
Suppose $e_1,e_2$ are the other two edges from $p$.
Let $A$ be the set of leaves of $T$ lying to the side of $e_1$ opposite to $p$ 
and $B$ be the set of leaves of $T$ lying to the side of $e_2$ opposite to $p$.
Then we require $A\subset I$ and $B\subset J$. 
It is easy to see that such $e_0$ exists because $I\sqcup J$ is not a $T$-partition.
We set $a=\sum_{i\in A} d_i$ and $b=\sum_{j\in B} d_j$.

Let $\iota$ be the involution on the set of the planar embeddings of $T$ given
by the change of orientation of the edges $e_0,e_1,e_2$ at $p$. 
Suppose $\sigma$ is a cyclic ordering corresponding to some planar embedding of $T$.
We denote by $\iota(\sigma)$ the cyclic ordering corresponding to the embedding of $T$
obtained by applying $\iota$ to the planar embedding of $\sigma$. 
Recall the definition of $\sigma_I(\vec{d})$ from Definition \ref{D:sigma-vector}.
The key observation now is that whenever $\sigma_I(\vec{d})=(z_0,\dots, z_{m-1})$,
we have 
\begin{equation*}
\iota(\sigma)_I(\vec{d})=(z_0,\dots, z_{m-1})+(-1)^\epsilon (\underbrace{1,\dots, 1}_{a}, 
\underbrace{0,\dots,0}_{b}, 0, \dots)+(-1)^{\epsilon+1}(\underbrace{0,\dots,0}_{b}\underbrace{1,\dots, 1}_{a}, 0, \dots),
\end{equation*}
where $\epsilon\in\{0,1\}$ depends on the original orientation of $\{e_0,e_1,e_2\}$ at $p$. 

To prove the lemma, it remains to show that the average of $w(I\mid J)$-flows
for any pair $(\sigma, \iota(\sigma))$ of embeddings is at least $\widetilde{f}(0)$.

Suppose that neither $\sigma_I(\vec{d})$ nor $\iota(\sigma)_I(\vec{d})$ is $(0,\dots, 0)$.
Then because $Q_f$ satisfies condition ($\dagger$) of Proposition \ref{P:new-nef}, we have 
both $w_{\sigma}(I\mid J)\geq \widetilde{f}(0)$ and $w_{\iota(\sigma)}(I\mid J) \geq \widetilde{f}(0)$
as desired.

Suppose now that $\sigma_I(\vec{d})=(0,\dots, 0)$. 
Then $\iota(\sigma)_I(\vec{d})=(\underbrace{1,\dots,1}_{b}, \underbrace{0,\dots,0}_{a-b}, \underbrace{-1,\dots,-1}_{b},0,\dots,0)$.
It follows that
\begin{multline*}
\frac{1}{2}\left(w_{\sigma}(I\mid J)+w_{\iota(\sigma)}(I\mid J)\right) =\frac{1}{2}Q_f\bigl(\underbrace{1,\dots,1}_{b}, \underbrace{0,\dots,0}_{a-b}, \underbrace{-1,\dots,-1}_{b}, 0 \dots, 0\bigr)\\ =
 2f(a)+2f(b)-f(a+b)-f(a-b) \geq \widetilde{f}(0),
\end{multline*}
as desired.
\end{proof}

At last, the $T$-effectivity of $\widetilde{f}$ with respect to $\vec{d}$ follows
from Properties (1--3) of the $T$-cyclic weighting $w_T$.
\end{proof}

\begin{example}\label{E:new} 
Suppose $m\geq 3$. The $1^{st}$ standard function $A_m\co \ZZ_m \to \ZZ$ has
the associated quadratic form $Q_{A_m}=m(x_0^2+\cdots+x_{m-1}^2)-(x_0+\cdots+x_{m-1})^2$
that satisfies the assumptions of Proposition \ref{P:new-nef}.
It is straightforward to compute that $m(A_m)=m$. As a result, Theorem \ref{T:new-nef}
shows that the $3^{rd}$ standard function from Example \ref{example-E} is effective. 
This gives a streamlined proof
of Theorem 4.2 from unpublished \cite{new-nef} stating that the following divisor is nef
on $\M_{0,n}$ for any $\ZZ_m$-$n$-tuple $\vec{d}=(d_1,\dots,d_n)$:
\begin{equation}\label{E:new-nef}
\begin{aligned}
\DD\bigl(\ZZ_m, E_m; \vec{d}\bigr) &=\sum_{i=1}^n \modp{d_i}{m}\modp{m-d_i}{m}\psi_i
-\sum_{I\sqcup J=[n]} \modp{\vec{d}(I)}{m}\modp{\vec{d}(J)}{m}\Delta_{I,J} \\
&+m\Bigl(\sum_{\substack{i: \\ m\mid d_i}} \psi_i-\sum_{\substack{I\sqcup J=[n]\\   m \mid \vec{d}(I)}} \Delta_{I,J}\Bigr).
\end{aligned}
\end{equation}

By taking $n=9$, $d_1=\cdots=d_9=1$, and $m=3$, we obtain the divisor $\Delta_2+\Delta_3+2\Delta_4$ in $\Nef(\M_{0,9})$.
This divisor generates an extremal ray of the symmetric nef cone of $\M_{0,9}$ and is not known to come from 
the conformal blocks bundles \cite{swinarski}, or any other geometric construction. 
This is also the only extremal ray of $\Nef(\M_{0,9}/\Sm_9)$ not known to be semiample, see Table \ref{table-experiments}
in  \S\ref{S:experiments}.

It is proved in \cite{fedorchuk-cyclic}, that in the case $d_1=\cdots=d_n$ and $m=3$, the divisor 
\eqref{E:new-nef} generates an extremal ray of the symmetric nef cone of $\M_{0,n}$. 
We believe the following:
\begin{conjecture}
Suppose $m\geq 5$ is prime, $m \mid n$ and $d_1=\cdots=d_n$. 
Then the divisor \eqref{E:new-nef} generates an extremal ray of the symmetric nef cone
of $\M_{0,n}$.
\end{conjecture}
We do not know whether divisors given by Theorem \ref{T:new-nef} 
are in general semiample.  
\end{example}

\section{Applications and examples}
\label{S:examples}

\subsection{Applications to conformal blocks divisors}
\label{S:examples-CB}
We have already seen in Example \ref{E:typeA} that 
the semiampleness of all $\sl_m$ level $1$ conformal blocks divisors studied in \cite{gian-gib}
follows from the balancedness of the quadratic form  
\[
x_0^2+x_1^2+\cdots+x_{m-1}^2.
\]
In fact, Theorem \ref{P:semiample} implies that all such divisors are boundary semiample over $\spec \ZZ$.
By \cite{kazanova}, the above result implies that all symmetric type A conformal blocks vector bundles
of rank $1$ are also boundary semiample over $\spec \ZZ$. 

There are several other notable connections of our results with the theory of conformal blocks
divisors on $\M_{0,n}$. For example, by applying the Effective Boundary Lemma \ref{L:main},
Mukhopadhyay showed that all level $1$ conformal blocks divisors of type B and D are effective boundary
\cite{swarnava}. It would be interesting to determine whether the associated quadratic forms 
of these divisors enjoy any balancedness properties. 

In \cite{GJMS}, Gibney, Jensen, Moon, and Swinarski 
introduce and study Veronese quotient divisors on $\M_{0,n}$ and relate
them to higher level $\sl_2$ conformal blocks divisors. 
For example, for every positive integer $g$ and an integer $1\leq \ell \leq g$, they consider
the following Veronese quotient divisors on $\M_{0,2g+2}$ \cite[Example 2.13]{GJMS}:
\[
D(g,\ell):=\mathcal{D}_{\frac{\ell-1}{\ell+1}, \left(\frac{1}{\ell+1}\right)^{g+1}}.
\]
It follows from the formulae for the intersection numbers of this divisor with the F-curves
$F_{1,1,a}$ given in \cite[Corollary 2.2]{GJMS} and Remark \ref{R:F-basis} that the associated quadratic form $q_{D(g,\ell)}$ is weakly balanced
if and only if $\ell$ is odd. Therefore, we conclude that 
$D(g,\ell)$ is cyclic semiample for all odd $\ell$. 
Applying \cite[Corollary 4.5]{GJMS}, we deduce that the $\sl_2$ level $\ell$ conformal blocks divisor 
$\mathbb{D}(\sl_2, \ell, \omega_1^{2g+2})$ is cyclic semiample on $\M_{0,n}$ for all odd $\ell$.

Computations show that the F-nef function $f_{D(g,\ell)}$ often satisfies tree-effectivity 
criterion of Proposition \ref{P:tree-effective} for even values of $\ell$. Determining whether $f_{D(g,\ell)}$ is 
tree-effective for these divisors would be another interesting problem.

\subsection{Democratic weighting}
While the major focus of this paper has been on proving semiampleness of 
divisors by proving tree-effectivity of the associated function, it is often easier
to verify effectivity (or weak-effectivity) of a function. One obvious way to do this is as follows.

Given a function $f\co G \ra \QQ$ and a $G$-$n$-tuple $(d_1,\dots,d_n)$,
we define the \emph{democratic weighting} $w_{\dem}$ on $\Gamma([n])$ by the following recipe.
Set
\[
\Sigma:=\sum_{i=1}^{n} f(d_i),
\]
and 
\[
c_i:=\frac{1}{n-2}f(d_i)- \frac{1}{2(n-1)(n-2)}\Sigma.
\]
We now set $w_{\dem} (i\sim j)=c_i+c_j$. It is easy to check
that the flow through each vertex $i$ is $f(d_i)$. 
Next, consider a partition $I\sqcup J=[n]$ with $\card(I)=k$ and $\card(J)=n-k$.
The $w$-flow across $I\sqcup J$ is then
\begin{multline*}
\frac{k}{n-2}\sum_{j\in J} f(d_j)+\frac{n-k}{n-2}\sum_{i\in I} f(d_i) -\frac{k(n-k)}{(n-1)(n-2)}
\Sigma \\
=\left(\frac{k}{n-2}-\frac{k(n-k)}{(n-1)(n-2)}\right)\sum_{j\in J} f(d_j)
+\left(\frac{n-k}{n-2}-\frac{k(n-k)}{(n-1)(n-2)}\right)\sum_{i\in I} f(d_i) \\
=\frac{k(k+1)}{(n-2)(n-1)} \sum_{j\in J} f(d_j)
+\frac{(n-k)(n-k-1)}{(n-1)(n-2)}\sum_{i\in I} f(d_i). 
\end{multline*}

\begin{prop}\label{P:easy}
Suppose $f\co G \ra \QQ$ is an F-nef function such that 
\[
f(a)+f(b) \geq \frac{3}{2}f(a+b) \quad \text{for all $a,b \in G$}.
\]
Then $f$ is an effective function. In particular,
$\DD\bigl(G, f; (d_1,\dots, d_n)\bigr)$ is nef on $\M_{0,n}$ 
for all $G$-$n$-tuples $(d_1,\dots,d_n)$.
\end{prop}

\begin{proof}
It is easy to see using the democratic weighting that $f$ is effective with respect
to any $G$-$n$-tuple. The claim follows by Theorem \ref{P:semiample}.
\end{proof}

\begin{definition}
We say that a divisor $D=\sum_{r=2}^{\nhalf} a_r \Delta_r$ on $\M_{0,n}$
is \emph{democratic} if for some symmetric function $f\co \ZZ_n \to \QQ_{\geq 0}$ 
such that $D=\DD(\ZZ_n, f; (1)_n)$ and for all $k$-tuples of positive integers $\vec{d}=(d_1,\dots,d_k)$
satisfying $d_1+\cdots+d_k=n$, the democratic weighting proves that $f$ is effective 
with respect to $\vec{d}$. Note that by Theorem \ref{P:weakly-semiample}, a democratic divisor is stratally effective boundary, hence nef. 
\end{definition}

\subsection{Experimental results}
\label{S:experiments}
We wrote a simple worksheet in SAGE \cite{sage} (which in turn uses PARI/GP \cite{pari} for certain computations
with quadratic forms), available at \\ \url{https://www2.bc.edu/maksym-fedorchuk/Semiampleness-criterion.sws},
\\ that can compute, among other things, whether:
\begin{itemize}
\item A positive definite cyclic quadratic form is good.\footnote{Unfortunately, SAGE or PARI/GP 
quickly runs out of memory for more interesting quadratic forms.} 
\item A given symmetric divisor on $\M_{0,n}$ satisfies semiampleness 
criterion of Theorem \ref{T:semiample}.
\item A given symmetric divisor on $\M_{0,n}$ satisfies semiampleness 
criterion of Theorem \ref{T:2nd-semiample}.
\item A given symmetric divisor on $\M_{0,n}$ is democratic, hence nef. 
\end{itemize}

We compute the number of extremal rays of the symmetric F-cone of $\M_{0,n}$ 
that are semiample by Theorems \ref{T:semiample} and \ref{T:2nd-semiample} in Table \ref{table-experiments}.
For the remaining extremal rays, we check whether they are democratic, hence nef. 
The output for a given $n$ is obtained by running the command {\tt SemiampleTest($n$)} of the worksheet.

\begin{table}[th] 
\renewcommand{\arraystretch}{1.1}
\begin{tabular}{|c|c|c|c|c|}
\hline
\multirow{4}{0.3in}{\ \ $n$  } & \multirow{4}{1.1in}{\footnotesize{Extremal rays of the symmetric F-cone} } & 
\multirow{4}{1.3in}{\footnotesize{Cyclic semiample,} \ i.e., satisfy\ Theorem \ref{T:semiample}\ } &   
\multirow{4}{1.4in}{\footnotesize{Satisfy Theorem \ref{T:2nd-semiample}, \ but not cyclic semiample}} &   
\multirow{4}{1.3in}{\footnotesize{Democratic divisors among the remaining extremal rays }}\\
& & & &  \\
& & & & \\
& & & & \\ 
\hline
$8$ & $4$ & $3$ & $1$ & $0$ \\
$9$ & $4$ & $3$ & $0$ & $1$
\\
$10$ & $7$ & $6$ & $1$ & $0$ \\
$11$ & $10$ & $6$ & $0$ & $4$ \\
$12$ & $10$ & $6$ & $1$ & $3$ \\
$13$ & $18$ & $9$ & $0$ & $9$ \\
$14$ & $27$ & $13$ & $1$ & $13$ \\
$15$ & $26$ & $11$ & $0$ & $15$\\
$16$ & $74$ & $19$ & $7$ & $48$\\
$17$ & $113$ & $22$ & $0$ & $84$\\

\hline
\end{tabular}
\smallskip

\caption{Experimental results}
\label{table-experiments}
\end{table}

Referring to the table above, we obtain the following result, which extends
the validity of the symmetric F-conjecture to positive characteristic for $n\leq 16$ (recall that 
in characteristic $0$, it is known for $n\leq 24$ by \cite{Gib}):
\begin{theorem} Over $\spec \ZZ$, we have that every symmetric F-nef divisor on $\M_{0,n}$  
is stratally effective boundary (hence nef over any field) for $n\leq 16$.
Moreover, all symmetric nef divisors on $\M_{0,8}$ and $\M_{0,10}$ are semiample.
\end{theorem}

\subsection*{Acknowledgements.} 
The author was partially supported by NSF grant DMS-1259226 and a Sloan Research Fellowship.
We thank Ana-Maria Castravet, Anand Deopurkar, Anand Patel,
and Jenia Tevelev for stimulating discussions. We also thank  
Ian Morrison for generous comments and suggestions that improved readability of
the paper. A portion of this work was completed when the author visited 
the Max Planck Institute for Mathematics in Bonn in June 2014.

\newcommand{\netarxiv}[1]{\href{http://arxiv.org/abs/#1}{{\sffamily{\texttt{arXiv:#1}}}}}


\bibliographystyle{alpha} 
\bibliography{ref-semiample}

\newcommand{\etalchar}[1]{$^{#1}$}
\begin{thebibliography}{GJMS13}

\bibitem[AC98]{AC}
Enrico Arbarello and Maurizio Cornalba.
\newblock Calculating cohomology groups of moduli spaces of curves via
  algebraic geometry.
\newblock {\em Inst. Hautes \'Etudes Sci. Publ. Math.}, (88):97--127 (1999),
  1998.

\bibitem[AGS14]{ags}
Valery Alexeev, Angela Gibney, and David Swinarski.
\newblock Higher-level {$\mathfrak{sl}_2$} conformal blocks divisors on
  {$\overline M_{0,n}$}.
\newblock {\em Proc. Edinb. Math. Soc. (2)}, 57(1):7--30, 2014.

\bibitem[AGSS12]{agss}
Maxim Arap, Angela Gibney, James Stankewicz, and David Swinarski.
\newblock {$sl_n$} level 1 conformal blocks divisors on {$\overline M_{0,n}$}.
\newblock {\em Int. Math. Res. Not. IMRN}, (7):1634--1680, 2012.

\bibitem[AS12]{alexeev-swinarski}
Valery Alexeev and David Swinarski.
\newblock Nef divisors on {$\overline M_{0,n}$} from {GIT}.
\newblock In {\em Geometry and arithmetic}, EMS Ser. Congr. Rep., pages 1--21.
  Eur. Math. Soc., Z\"urich, 2012.

\bibitem[BGM13]{conf-block}
Prakash Belkale, Angela Gibney, and Swarnava Mukhopadhyay.
\newblock Critical level vanishing and identities on $\overline{M}_{0,n}$,
  2013.
\newblock {\tt arXiv:1308.4906 [math.AG]}, to appear in \emph{Algebraic
  Geometry}.

\bibitem[Car09]{carr-2009}
Sarah Carr.
\newblock A polygonal presentation of {$\Pic(\overline{M}_{0,n})$}, 2009.
\newblock {\tt arXiv:0911.2649 [math.AG]}.

\bibitem[CT13]{castravet-tevelev-MDS}
Ana-Maria Castravet and Jenia Tevelev.
\newblock $\overline{M}_{0,n}$ is not a {M}ori {D}ream {S}pace, 2013.
\newblock {\tt arXiv:1311.7673 [math.AG]}, to appear in \emph{Duke Math. J.}

\bibitem[Fak12]{fakh}
Najmuddin Fakhruddin.
\newblock Chern classes of conformal blocks.
\newblock In {\em Compact moduli spaces and vector bundles}, volume 564 of {\em
  Contemp. Math.}, pages 145--176. Amer. Math. Soc., Providence, RI, 2012.

\bibitem[Fed11]{fedorchuk-cyclic}
Maksym Fedorchuk.
\newblock Cyclic covering morphisms on {$\overline{M}_{0,n}$}, 2011.
\newblock {\tt arXiv:1105.0655 [math.AG]}.

\bibitem[Fed13]{new-nef}
Maksym Fedorchuk.
\newblock New nef divisors on $\overline{M}_{0,n}$, 2013.
\newblock {\tt arXiv:1308.5993 [math.AG]}.

\bibitem[FG03]{FG}
Gavril Farkas and Angela Gibney.
\newblock The {M}ori cones of moduli spaces of pointed curves of small genus.
\newblock {\em Trans. Amer. Math. Soc.}, 355(3):1183--1199 (electronic), 2003.

\bibitem[GG12]{gian-gib}
Noah Giansiracusa and Angela Gibney.
\newblock The cone of type {$A$}, level 1, conformal blocks divisors.
\newblock {\em Adv. Math.}, 231(2):798--814, 2012.

\bibitem[Gia13]{gian}
Noah Giansiracusa.
\newblock Conformal blocks and rational normal curves.
\newblock {\em J. Algebraic Geom.}, 22(4):773--793, 2013.

\bibitem[Gib09]{Gib}
Angela Gibney.
\newblock Numerical criteria for divisors on {$\overline M_g$} to be ample.
\newblock {\em Compos. Math.}, 145(5):1227--1248, 2009.

\bibitem[GJMS13]{GJMS}
Angela Gibney, David Jensen, Han-Bom Moon, and David Swinarski.
\newblock Veronese quotient models of {$\overline{\rm M}_{0,n}$} and conformal
  blocks.
\newblock {\em Michigan Math. J.}, 62(4):721--751, 2013.

\bibitem[GK14]{gonzalez-karu}
Jos\'{e}~Luis Gonz\'{a}lez and Kalle Karu.
\newblock Some non-finitely generated {Cox} rings, 2014.
\newblock {\tt arXiv:1407.6344 [math.AG]}.

\bibitem[GKM02]{GKM}
Angela Gibney, Sean Keel, and Ian Morrison.
\newblock Towards the ample cone of {$\overline M_{g,n}$}.
\newblock {\em J. Amer. Math. Soc.}, 15(2):273--294 (electronic), 2002.

\bibitem[HK00]{hu-keel}
Yi~Hu and Se{\'a}n Keel.
\newblock Mori dream spaces and {GIT}.
\newblock {\em Michigan Math. J.}, 48:331--348, 2000.
\newblock Dedicated to William Fulton on the occasion of his 60th birthday.

\bibitem[Kap93]{kap1}
M.~M. Kapranov.
\newblock Chow quotients of {G}rassmannians. {I}.
\newblock In {\em I. {M}. {G}el'fand {S}eminar}, volume~16 of {\em Adv. Soviet
  Math.}, pages 29--110. Amer. Math. Soc., Providence, RI, 1993.

\bibitem[Kaz14]{kazanova}
Anna Kazanova.
\newblock On {$S_n$}-invariant conformal blocks vector bundles of rank one on
  {$\overline{M}_{0,n}$}, 2014.
\newblock {\tt arXiv:1404.5845 [math.AG]}.

\bibitem[Kee92]{keel}
Sean Keel.
\newblock Intersection theory of moduli space of stable {$n$}-pointed curves of
  genus zero.
\newblock {\em Trans. Amer. Math. Soc.}, 330(2):545--574, 1992.

\bibitem[KM96]{KMc}
Se{\'a}n Keel and James McKernan.
\newblock Contractible extremal rays on {$\overline M_{0,n}$}.
\newblock In {\em Handbook of moduli. {V}ol. {II}}, volume~25 of {\em Adv.
  Lect. Math. (ALM)}, pages 115--130. Int. Press, Somerville, MA, 1996.

\bibitem[Lar12]{larsen}
Paul~L. Larsen.
\newblock Fulton's conjecture for {$\overline M_{0,7}$}.
\newblock {\em J. Lond. Math. Soc. (2)}, 85(1):1--21, 2012.

\bibitem[Mor07]{ian-mori}
Ian Morrison.
\newblock Mori theory of moduli spaces of stable curves, 2007.
\newblock Projective Press, New York.

\bibitem[Muk14]{swarnava}
Swarnava Mukhopadhyay.
\newblock Remarks on level-one conformal blocks divisors.
\newblock {\em C. R. Math. Acad. Sci. Paris}, 352(3):179--182, 2014.

\bibitem[Pix13]{pixton}
Aaron Pixton.
\newblock A nonboundary nef divisor on {$\overline{M}_{0,12}$}.
\newblock {\em Geom. Topol.}, 17(3):1317--1324, 2013.

\bibitem[S{\etalchar{+}}14]{sage}
W.\thinspace{}A. Stein et~al.
\newblock {\em {S}age {M}athematics {S}oftware ({V}ersion 6.0)}.
\newblock The Sage Development Team, 2014.
\newblock {\tt http://www.sagemath.org}.

\bibitem[SS04]{sturmfels-speyer}
David Speyer and Bernd Sturmfels.
\newblock The tropical {G}rassmannian.
\newblock {\em Adv. Geom.}, 4(3):389--411, 2004.

\bibitem[{Swi}11]{swinarski}
David {Swinarski}.
\newblock {{$sl_2$} conformal block divisors and the nef cone of
  {$\overline{M}_{0,n}$}}, 2011.
\newblock {\tt arXiv:1107.5331 [math.AG]}.

\bibitem[{The}14]{pari}
{The PARI~Group}, Bordeaux.
\newblock {\em {PARI/GP version {\tt 2.7.0}}}, 2014.
\newblock available from \url{http://pari.math.u-bordeaux.fr/}.

\bibitem[Vak03]{vakil-taut}
Ravi Vakil.
\newblock The moduli space of curves and its tautological ring.
\newblock {\em Notices Amer. Math. Soc.}, 50(6):647--658, 2003.

\end{thebibliography}

\end{document}